\documentclass[11pt]{amsart}
\usepackage{hyperref}
\usepackage{epsf}
\usepackage{amsmath, accents,amsfonts, amscd, amssymb, amsthm, latexsym, verbatim,
dsfont,enumerate,url,}
\usepackage{multicol,times}
\usepackage{color}
\usepackage{todonotes}
\usepackage[all]{xy}  
\usepackage[normalem]{ulem}
\usepackage{fullpage}

\usepackage{amsthm,amsfonts,amsmath,amscd,amssymb,latexsym,
pb-diagram,amssymb,epic,eepic,verbatim,graphicx,graphics,epsfig,psfrag} 
\usepackage{mathrsfs}
\usepackage{mathdots}

\def\N{{\mathbb N}}
\def\R{\mathbb{R}}
\def\Z{{\mathbb Z}}
\def\C{{\mathbb C}}
\def\B{{\mathbb B}}
\def\E{{\mathbb E}}
\def\F{{\mathbb F}}

\def\W{{\mathbb W}}

\newcommand{\CP}{\mathbb{C}\mathbb{P}}

\def\b{\beta}
\def\d{\delta}

\newcommand{\eps}{\epsilon}
\def\s{\sigma}

\def\g{\gamma}

\def\s{\sigma}

\def\l{\lambda}

\def\Si{\Sigma}

\def\cB{{\mathcal B}}
\def\cC{{\mathcal C}}

\def\cE{{\mathcal E}}

\def\cL{{\mathcal L}}

\def\cU{{\mathcal U}}

\def\ud{{\underline{\delta}}}
\def\uxi{{\underline{\xi}}}

\def\rD{{\rm D}}
\def\rT{{\rm T}}
\def\rd{{\rm d}}

\def\dt{{\rm d}t}

\newcommand{\CR}{\bar{\partial}}
\newcommand{\pbar}{\bar{\partial}_{J,X}}

\def\bM{{\overline{\mathcal M}}}
\newcommand{\CM}{\overline{\mathcal{M}}}

\def\tsum{\textstyle\sum}

\def\less{{\smallsetminus}}        

\newcommand{\ti}{\tilde}
\newcommand{\Ti}{\widetilde}
\newcommand\ul{\underline}     

\def\st{\: \big| \:}                  

\newcommand\quotient[2]{
        \mathchoice
            {
                \text{\raise1ex\hbox{$#1$}\Big/\lower1ex\hbox{$#2$}}
            }
            {
                #1\,/\,#2
            }
            {
                #1\,/\,#2
            }
            {
                #1\,/\,#2
            }
    }
\newcommand\quo[2]{
                \text{\raise1ex\hbox{$#1\!$}\big/\lower1ex\hbox{$\!#2$}}
  }

\DeclareMathOperator{\Hom}{Hom}

\DeclareMathOperator{\pr}{pr}

\DeclareMathOperator{\supp}{supp}

\newcommand{\id}{\operatorname{id}}

\newcommand{\ind}{\operatorname{ind}}

\newcommand{\im}{\operatorname{im}}

\def\mi{{\raisebox{1pt}{$\scriptscriptstyle{-}$}}}
\def\pl{{\raisebox{1pt}{$\scriptscriptstyle{+}$}}}
\def\pmi{{\raisebox{1pt}{$\scriptscriptstyle\pm$}}}
\def\mip{{\raisebox{1pt}{$\scriptscriptstyle\mp$}}}

\newcommand\qu[2]{
                \text{\raise.8ex\hbox{$\scriptstyle#1$}/\lower.8ex\hbox{$\scriptstyle#2$}}
  }

\newcounter{qcounter}

\newenvironment{enumlist}
   {
      \begin{list}
         {(\roman{qcounter})}
         {
         \usecounter{qcounter}
                     \setlength{\itemsep}{.4ex}
            \setlength{\leftmargin}{1em}
         }
   }
   {
      \end{list}
   }

\newenvironment{itemlist}
   { \begin{list} {$\bullet$}
         {  \setlength{\itemsep}{.4ex}
            \setlength{\leftmargin}{1em} } }
   { \end{list} }

\newtheorem{theorem}{Theorem}[section]

\newtheorem{lemma}[theorem]{Lemma}

\newtheorem{definition}[theorem]{Definition}
\newtheorem{remark}[theorem]{Remark}

\begin{document}


\title{Fredholm notions in scale calculus and Hamiltonian Floer theory}
\author{Katrin Wehrheim}

\begin{abstract}
We give an equivalent definition of the Fredholm property for linear operators on scale Banach spaces and
introduce a nonlinear scale Fredholm property with respect to a splitting of the domain. The latter implies the Fredholm property introduced by Hofer-Wysocki-Zehnder in terms of contraction germs, but is easier to check in practice and holds in applications to holomorphic curve moduli spaces.
We demonstrate this at the example of trajectory breaking in Hamiltonian Floer theory, using a setup that can also be specified to Morse theory.
\end{abstract}

\maketitle

\tableofcontents

\section{Introduction}
Scale calculus was developed by H.~Hofer, K.~Wysocki and E.~Zehnder as part of polyfold theory, which provides an analytic framework for the study of moduli spaces of pseudo-holomorphic curves. 
Roughly speaking, such moduli spaces are (compactifications of) sets of equivalence classes of smooth maps which satisfy the Cauchy-Riemann equation, where two maps \(u\) and \(v\) are equivalent provided there exists a holomorphic automorphism \(\phi\) of the domain such that \(u=v\circ \phi\).  
Since these spaces are studied for almost complex structures, there is no readily available algebraic framework, so that they instead are viewed as solution spaces to a nonlinear PDE, modulo a reparametrization action by a usually finite dimensional Lie Group. The fundamental analytic difficulty in this setup is that Fredholm theory for the Cauchy-Riemann operator requires a Banach space completion of the space of smooth maps, while reparametrizations do not act differentiably in any known Banach completion (see \S\ref{s:calc} and \cite{theguide,McDuffWehrheim}).

The novel approach of polyfold theory to this issue is to replace the classical notion of differentiability in Banach spaces by a new notion of scale differentiability on scale Banach spaces, which allows for a natural framework of a scale of Sobolev spaces (a sequence indexed by the differentiability of the maps), in which reparametrizations act scale smoothly.
The resulting scale calculus for scale manifolds is rich enough to establish a regularization theorem for suitably defined scale Fredholm sections with compact zero set, which allows to associate cobordism classes of smooth manifolds to e.g.\ compact moduli spaces of pseudo-holomorphic curves which contain no nodal curves.
To deal with the latter, polyfold theory introduces a second fundamentally new concept -- generalizing the local models for Banach manifolds to images of scale smooth retractions. 


A brief guide to the notions of polyfold theory used in this paper can be found Appendix \ref{sec:polyfold}. More in-depth introductions to polyfold theory can be found in \cite{hwzbook}, the surveys \cite{hoferCD,theguide}, and its applications in \cite{hwz:gw,fh-sft,BDSVW}.

This note aims to shed some light on the abstract linear and nonlinear Fredholm theory on scale Banach spaces, which obviously are crucial ingredients of the abstract regularization approach provided by polyfold theory.
In particular, our goal is to explain at examples how the classical Fredholm property of elliptic PDE's in many cases directly implies a scale Fredholm property -- at least in the absence of singularities such as nodes.
We give a quick overview of the basic definitions of scale calculus in \S\ref{s:calc} in order to fix notation and make this note (with the exception of the proofs and \S\ref{s:Ham}) self-contained. In particular, we introduce the notion of a {\it norm scale} on an infinite dimensional vector space, which -- by completion -- is the source of all scale Banach spaces.
\S\ref{s:lin} discusses the notion of a linear Fredholm operator on scale Banach spaces, as warmup for the nonlinear theory. We introduce a simpler definition, demonstrate why it is naturally satisfied by elliptic operators, and prove that it is equivalent to the definition in \cite{hwzbook}.
As core of the paper, \S\ref{s:nonlin} then introduces a simpler nonlinear Fredholm property with respect to a splitting in Definition~\ref{def:scfred}, which in practice will be given by splitting off a finite dimensional space of gluing parameters from an infinite dimensional function space. We show in Theorem~\ref{thm} that this Fredholm property implies the Fredholm property based on a contraction germ normal form that is introduced in \cite{hwzbook}.
We moreover explain in Remark~\ref{rmk:contract} the need for a separate nonlinear Fredholm notion in scale calculus, whereas the classical notion of a nonlinear Fredholm map is given simply by requiring continuous differentiability and the linear Fredholm property for the linearized maps (at zeros).

Finally, \S\ref{s:Ham} uses the example of Hamiltonian Floer theory to demonstrate the strength of our simplified nonlinear Fredholm notion and polyfold theory in general.
While we defer the global setup in terms of retractions and fillings to \cite{fh-sft-full}, we give a complete proof of the Fredholm property for Floer's perturbed Cauchy-Riemann operator near a broken Floer trajectory in Theorem~\ref{mainthm}.
Combined with the abstract transversality and implicit function theorem for polyfolds in \cite{hwzbook}, this Fredholm property replaces the entire transversality and gluing techniques in the classical construction of Hamiltonian Floer theory in the absence of bubbling, as developed by Floer \cite{floer}. It even provides a smooth structure with boundary and corners on perturbations of the compactified moduli spaces of Floer trajectories.\footnote{
The boundary stratification of this regularization arises from trajectory breaking exclusively; sphere bubbling yields interior codimension 2 strata. This reduces the challenge of constructing Floer homology to the combinatorial issue of ordering the choice of perturbations on different moduli spaces so that they can be made coherent with ``gluing identifications'' of the boundary strata with fiber products of previous moduli spaces.
}

For readers who are not familiar or not interested in Floer theory, this last section can also be read as a polyfold setup for Morse theory. However, the Morse-Witten moduli spaces do not require polyfold technology for regularization; see e.g.\ \cite{schwarz:Morse,W:Morse}.

\smallskip
\noindent
{\bf Acknowledgements:}
This work was inspired by questions arising in \cite{theguide} with with Joel Fish, Roman Golovko, and Oliver Fabert, and was supported by NSF grant DMS 0844188 and the Institute for Advanced Studies.
Peter Albers, Nate Bottman, Helmut Hofer, Jiayong Li, Dusa McDuff, and Chris Policastro provided helpful discussions and detection of a number of gaps in earlier versions.
The presentation and precision was also improved by two very helpful referee reports.

\section{Some basic scale calculus} \label{s:calc}

We begin with a new definition, which in practice is applied to spaces of smooth maps to give rise to scale Banach spaces.

\begin{definition}\label{def:scCompletion}
A {\bf norm scale} on a vector space $F$ is a sequence ${(\|\cdot\|_k)_{k\in \N_0}}$ of norms on $F$ such that for each $k>j$ the identity map 
$$
\id_{F} \,:\; (F,\|\cdot\|_k) \;\longrightarrow\; (F,\|\cdot\|_j)
$$ 
is continuous and compact. That is,
\begin{itemize}
\item 
there is a constant $C_{k,j}$ such that $\|f\|_j \leq C_{k,j} \|f\|_k$ for all $f\in F$,
\item
any $\|\cdot\|_k$-bounded sequence $(f_n)_{n\in\N}$  (that is with $\sup_{n\in\N} \| f_n\|_k <\infty$) has a subsequence $(f_{n_i})_{i\in\N}$ that is a Cauchy sequence in $(F,\|\cdot\|_j)$ (that is with $\sup_{\ell>i} \| f_{n_\ell} - f_{n_i} \|_j \to 0$ for $i\to\infty$).
\end{itemize}
\end{definition}

The compactness property of $\id_{F} : (F,\|\cdot\|_k) \to (F,\|\cdot\|_j)$ as stated above is equivalent to compactness of the map $\id_{F} : \overline{F}^{\|\cdot\|_k}  \to \overline{F}^{\|\cdot\|_j} $ between the completions: The image of any bounded set in $\overline{F}^{\|\cdot\|_k}$ is mapped to a precompact  set (i.e.\ its closure is compact) in $\overline{F}^{\|\cdot\|_j}$. 
On a finite dimensional vector space $F$, this compactness condition is automatic since the bounded sets have compact closure. 
Since moreover all norms on a finite dimensional vector space are complete and equivalent, there is no nontrivial example of a norm scale on a finite dimensional vector space. 
The following list provides the first two nontrivial examples together with the scale Banach spaces that result from completion in each scale, as formally introduced in Definition~\ref{def:scBanachSpace}.

\begin{itemlist}
\item
The $\cC^k$-norms $(\|\cdot\|_{\cC^k})_{k\in\N_0}$ are a norm scale on $F=\cC^\infty(S^1)$.
The $\|\cdot\|_{\cC^k}$-completions of $\cC^\infty(S^1)$ then form the scale Banach space 
$$
\F \;=\;\bigl( \, \overline{F}^{\|\cdot\|_k} \bigr)_{k\in\N_0}\;=\;\bigl( \cC^k(S^1) \bigr)_{k\in\N_0}.
$$
\item
The $W^{k,p}$ Sobolev-norms $(\|\cdot\|_{W^{k,p}})_{k\in\N_0}$  are a norm scale on $F=\cC^\infty(S^2,\C^n)$ for any $1\le p\le \infty$, $n\ge 1$.
The $\|\cdot\|_{W^{k,p}}$-completions of $\cC^\infty(S^2,\C^n)$ then form the scale Banach space 
$$
\F \;=\;\bigl( \, \overline{F}^{\|\cdot\|_k} \bigr)_{k\in\N_0}\;=\;\bigl( W^{k,p}(S^2,\C^n) \bigr)_{k\in\N_0}.
$$
\end{itemlist}

For the rest of this section we follow \cite{hwzbook} -- with some convenient tweaks of notation -- in developing the basic language of scale calculus.

\begin{definition}[\cite{hwzbook} 1.1] \label{def:scBanachSpace}
A {\bf sc-Banach space} $\E=(E_k)_{k\in \N_0}$ is a Banach space ${(E,\|\cdot\|)}$ together with an sc-structure $(E_k,\|\cdot\|_k)_{k\in \N_0}$, which consists of a sequence of linear subspaces $E_k\subset E$, each equipped with a Banach norm $\|\cdot\|_k$, such that the following holds.
\begin{enumerate}
\item We have $(E,\|\cdot\|)=(E_{0},\|\cdot\|_0)$ as Banach space.
\item 
For each $k>j$ there is an inclusion of subspaces $E_k\subset E_j$, and the inclusion map $(E_k,\|\cdot\|_k)\to (E_{j},\|\cdot\|_j)$ is continuous and compact.
\item 
The subspace $E_\infty:=\bigcap_{k\in\N_0} E_k\subset E$ is dense in each $(E_k,\|\cdot\|_k)$.
\end{enumerate}
\end{definition}

\begin{remark} \label{rmk:trivial}
The natural and in fact unique sc-structure on a finite dimensional vector space $E$ is the {\bf trivial sc-structure} $(E_k=E)_{k\in\N_0}$. 
In particular, for $n\in\N$ we denote by $\R^n$ and $\C^n$ the real and complex Euclidean spaces with standard norm and trivial sc-structure.
\end{remark}

At this point, the reader unfamiliar with scale calculus can get familiarized with the concept by checking that the completions with respect to a norm scale always form an sc-Banach space; in particular so do the examples above.
Next, the prototypical examples of scale smooth but not classically differentiable maps are the reparametrization actions in the above examples.

\begin{itemlist}
\item
The translation action on functions with domain $S^1:=\R/\Z$,
$$
\tau: \R \times \cC^0(S^1) \to \cC^0(S^1), \quad (s, f) \mapsto f(s + \cdot)
$$
has directional derivatives at points $(s_0,f_0)\in \R\times \cC^1(S^1)$, but the derivative $\frac\rd{\rd h}\bigr|_{h=0} (f_0+hF)(s_0+hS) = S \, \dot f_0(s_0 + \cdot) \;+\; F(s_0 + \cdot)$ in direction $(S,F)$ is not well defined for $f_0\in\cC^0(S^1)\less \cC^1(S^1)$.

Moreover, $\tau$ is in fact nowhere classically differentiable.\footnote{
The directional derivative in any fixed direction $(S,F)\in\R\times \cC^0(S^1)$ exists since uniform continuity of $F$ guarantees $\max_{s\in S^1} \bigl| F(s+h) - F(s) \bigr| \to 0$ as $h\to 0$. However, the unit ball in $\cC^0(S^1)$ is not equicontinuous, so that differentiability on the normed space, $\sup_{\|F\|_{\cC^0}=1} \max_{s\in S^1} \bigl| F(s+h) - F(s) \bigr| \to 0$ as ${h\to 0}$, fails.}
However, the restriction of $\tau$ to $\R \times \cC^{k+1}(S^1, \R) \to \cC^k(S^1)$ 
is continuously differentiable for any $k\in\N_0$. In fact, $\tau$ is sc$^\infty$ if we equip $\cC^0(S^1)$ with the sc-structure $(\cC^k(S^1))_{k\in\N_0}$. Its differential is
$$
\rD_{(s_0,f_0)}\tau (S, F) \;= \; S \, \dot f_0(s_0 + \cdot) \;+\; F(s_0 + \cdot) .
$$
(This example uses the product $\R\times\E$ of an sc-Banach space with the trivial sc-structure on $\R$, given by the scales $(\R\times E_k)_{k\in\N_0}$; c.f.\ Remark~\ref{rmk:prod} below.) 
\item
The reparametrization action by the group of M\"obius transformations $PSL(2,\C)$ on $S^2=\CP^1$,
$$
\theta: PSL(2,\C) \times W^{k,p}(S^2,\C^n) \to W^{k,p}(S^2,\C^n), \quad (\phi, u) \mapsto u\circ\phi
$$
has directional derivatives at $(\phi_0,u_0)\in PSL(2,\C)\times W^{k+1,p}(S^2,\C^n)$ but is differentiable only as map $PSL(2,\C) \times W^{k+1,p}(S^2,\C^n) \to W^{k,p}(S^2,\C^n)$.

However, $\theta$ is sc$^\infty$ on the sc-Banach space $\bigl( W^{k,p}(S^2,\C^n) \bigr)_{k\in\N_0}$.
(This example uses a scale structure on the nonlinear space $PSL(2,\C)$, which is defined via local charts of this Lie group in $\C^3$.)
\end{itemlist}

\begin{definition}[\cite{hwzbook} 1.1] \hspace{-2mm}\footnote{
Note that \cite{hwzbook} does not explicitly define a notion of scale differentiability as in (ii), but rather groups (ii)-(iv) for $k=1$ into the definition of continuous scale differentiability sc$^1$, which is the relevant notion for most purposes. The purpose of our definition (ii) is to define the tangent map (iii) in maximal generality.
}
\label{def:sc}
Let $\E,\F$ be sc-Banach spaces and let $\Phi:U\to F_0$ be a map defined on an open subset $U\subset E_0$.
\begin{enumlist}
\item
$\Phi$ is {\bf scale continuous} ($\mathbf{sc^0}$) if the restrictions $\Phi|_{U\cap E_m}: E_m\to F_m$ 
are continuous for all $m\in\N_0$;
\item
$\Phi$ is {\bf scale differentiable} if it is $sc^0$ and for every $x\in U\cap E_1$
there exists a bounded linear operator $\rD\Phi(x):E_0\to F_0$ such that
$$
\sup_{\|h\|_{E_1}=\hbar} \hbar^{-1}  \bigl\| \Phi( x + h ) - \Phi ( x ) - \rD\Phi(x) h \bigr\|_{F_0}
\underset{\hbar\to 0}\longrightarrow 0
$$
and $\rD\Phi(x) E_m \subset F_m$ whenever $x\in U\cap E_{m+1}$.
\item
If $\Phi$ is scale differentiable then its {\bf tangent map} is the map
$$
\rT\Phi: \rT\E|_U \to \rT\F , \qquad  (x,h) \mapsto (\Phi(x),\rD\Phi(x)h),
$$
defined on the open subset $\rT\E|_U:=(U\cap E_1)\times E_0$ of the sc-Banach space
$\rT\E:=( E_{m+1}\times E_m )_{m\in\N_0}$, mapping to
$\rT\F:=( F_{m+1}\times F_m )_{m\in\N_0}$.
\item
$\Phi$ is {\bf $\mathbf k$-fold continuously scale differentiable} ($\mathbf{sc^k}$) for $k\geq 1$
if it is scale differentiable and its tangent map $\rT\Phi:\rT\E|_U \to \rT \F$ is $sc^{k-1}$;
\item
$\Phi$ is {\bf scale smooth} ($\mathbf{sc^\infty}$) if it is $sc^k$ for all $k\in\N_0$.
\end{enumlist}
\end{definition}

Finally, we introduce a germ-like notion of scale smoothness at a point.

\begin{definition} \label{def:germsmooth}
Let $\E,\F$ be sc-Banach spaces and let $\Phi:U\to F_0$ be a map on an open set $U\subset E_0$. Then $\Phi$ is {\bf scale smooth at $\mathbf{e_0}\in U\cap E_\infty$}  (or $\mathbf{sc^\infty}$ {\bf at} $\mathbf{e_0}$) if for every $k\in\N_0$ there exists a neighbourhood $U_k\subset U$ of $e_0$ such that $\Phi|_{U_k}$ is sc$^k$. 
\end{definition}

\section{Fredholm property for linear operators} \label{s:lin}

We give a new definition of the Fredholm property for linear maps on scale Banach spaces, which in Lemma~\ref{sc is hwz} we show to be equivalent to the definition of \cite{hwzbook}.

\begin{definition} \label{fred op}
Let $\E, \F$ be sc-Banach spaces.
An {\bf sc-Fredholm operator} $T:\E\to\F$ is a linear map $T: E_0\to F_0$ that satisfies the following.
\begin{enumlist}
\item
$T$ is sc$^0$, that is all restrictions $T|_{E_m} : E_m \to F_m$ for $m\in\N_0$ are bounded.
\item
$T$ is {\bf regularizing}, that is $e\in E_0$ and $T e\in F_m$ for $m\in\N_0$ implies $e\in E_m$.
\item
$T: E_0 \to F_0$ is a Fredholm operator, that is it has finite dimensional kernel $\ker T$, closed range $T(E_0)$, and finite dimensional cokernel $F_0/T(E_0)$.
\end{enumlist}
\end{definition}

\noindent
The prototypical examples of sc-Fredholm operators are elliptic operators:

\begin{itemlist}
\item
Let $\E:=\bigl( \cC^{1+k}(S^1) \bigr)_{k\in\N_0}$ and $\F:=\bigl( \cC^{k}(S^1) \bigr)_{k\in\N_0}$, then 
$\frac\rd\dt : \cC^1(S^1) \to \cC^0(S^1)$ is an sc-Fredholm operator $\frac\rd\dt:\E\to\F$.
\item
The Cauchy--Riemann operator with respect to almost complex structures $j$ on $S^2$ and $J$ on $\C^n$
is an sc-Fredholm operator $\CR_J:\E\to\F$ for any $1<p<\infty$,
$$
\CR_J : W^{1,p}(S^2,\C^n) \to L^p(S^2,\Lambda^{0,1}\otimes_J\C^n)
, \quad
u\mapsto \tfrac 12( J \circ \rd u \circ j + \rd u ) , 
$$
where $\E:=\bigl( W^{1+k,p}(S^2,\C^n)  \bigr)_{k\in\N_0}$ 
and $\F:=\bigl( W^{k,p}(S^2,\Lambda^{0,1}\otimes_J\C^n) \bigr)_{k\in\N_0}$ are the $W^{k,p}$-closures of smooth $(J,j)$-antilinear $\C^n$-valued $1$-forms on $S^2$.
\end{itemlist}

The sc$^0$-property of these operators is a formalization of the fact that linear differential operators of degree $d$ are bounded as operators between appropriate function spaces with a difference of $d$ in the differentiability index.
The regularizing property is an abstract statement of elliptic regularity.
Finally, the elliptic estimates for an operator and its dual generally hold on all scales similar to the boundedness, which implies the Fredholm property on all scales.
Our sc-Fredholm notion formalizes the fact that one can obtain the Fredholm property of elliptic operators on all scales from their elliptic regularity together with the Fredholm property on a fixed scale, see Lemma~\ref{sc fred}.
The sc-Fredholm notion in \cite{hwzbook} is based on the following notion of direct sums in scale Banach spaces. 

\begin{definition} [\cite{hwzbook} 1.1]
Let $\E$ be an sc-Banach space.
Two linear subspaces $X,Y \subset E_0$ split $\E$ as an {\bf sc-direct sum} $\E = X \oplus_{sc} Y$ if
\begin{enumerate}
\item both $X$ and $Y$ are closed and carry scale structures $X\cap E_m$ resp.\ $Y\cap E_m$;
\item
on every level $m\in\N_0$ we have a direct sum $E_m = (X\cap E_m) \oplus (Y\cap E_m)$.
\end{enumerate}
If  $\, \E = X \oplus_{sc} Y$ then we call $Y$ the {\bf sc-complement} of $X$.
\end{definition}

\begin{remark} [\cite{hwz2} Glossary] \label{rmk:prod} 
There is a natural product notion $\E\times\F$ for sc-Banach spaces, such that $\E\times\F = (\E\times\{0\}) \oplus_{sc} (\{0\} \times \F)$. 
The {\bf sc-product} $\E\times\F$ of two sc-Banach spaces $\E,\F$ is the Cartesian product $E\times F$ with the scale structure $(E\times F)_k:=\bigl(E_k\times F_k , \|\cdot\|_{E_k} + \|\cdot\|_{F_k} \bigr)$.
\end{remark}

\begin{definition} [\cite{hwzbook} 1.1] \label{hwzfred}
Let $\E, \F$ be sc-Banach spaces.
A {\bf sc-Fredholm operator} $T:\E\to\F$ is a linear map $T: E_0\to F_0$ that satisfies the following.

\medskip
\noindent
(i)
The kernel $\ker T$ is finite dimensional with sc-complement
${\E = \ker T \oplus_{sc} X}$.

\medskip
\noindent
(ii)
The image $T(E_0)$ has a finite dimensional sc-complement
$\F = T(E_0) \oplus_{sc} C$.

\medskip
\noindent
(iii)
The operator restricts to an sc-isomorphism $T|_X : X \to T(E_0)$.
\end{definition}

Before proving equivalence of the two sc-Fredholm notions, we show that sc-Fredholm operators in the first sense are in fact Fredholm on each scale (as are the latter -- a simple consequence of the sc-direct sums).

\begin{lemma} \label{sc fred}
If $T: \E\to \F$ is sc-Fredholm as in Definition~\ref{fred op}, then the restrictions $T|_{E_m}: E_m \to F_m$ are Fredholm for all $m\in\N$ with kernel and cokernel
$$
\ker T|_{E_m} = \ker T  \subset E_\infty, \qquad  \quo{F_m}{T(E_m)} \cong \quo{F_0}{\im T} ,
$$
where the latter isomorphism is induced by the inclusion $F_m\subset F_0$.
In particular, the Fredholm index of $T|_{E_m}$ is the same on any scale $m\in\N_0$, 
$$
\ind(T) = \ind(T|_{E_m}) = \dim\ker T - \dim\bigl(\qu{F_m}{T(E_m)}\bigr) .
$$ 
\end{lemma}
\begin{proof}
Due to the embedding $E_m\subset E_0$, the kernel of $T|_{E_m}$ is $\ker T \cap E_m$, i.e.\ finite dimensional since it is a subspace of the kernel $\ker T\subset E_0$ that is finite dimensional by the Fredholm property of $T:E_0\to F_0$ given by (iii).
In fact, the regularization property (ii) for $0\in F_\infty$ implies $\ker T\subset E_\infty$, so that $\ker T \cap E_m=\ker T$ for all $m\in\N_0$.

Next, we will show that $T(E_m)\subset F_m$ is closed, although this will also follow from finite dimensionality of the cokernel.
For that purpose we need to consider any sequence $e_i\in E_m$ which has converging images $T(e_i)\to f_\infty$ in the $F_m$-topology, and show that $f_\infty\in T(E_m)$. Indeed, then closedness of $T(E_0)\subset F_0$ from the Fredholm property (iii) implies $f_\infty=T(e')$ for some $e'\in E_0$, and we have $T(e')=f_\infty\in F_m$ since it is the limit of a sequence in $F_m$, hence the regularization property (ii) implies that $e'\in E_m$ and hence $f_\infty=T(e')\in T(E_m)$ by the boundedness of $T|_{E_m}$ given by the sc$^0$ property (i).

The last part of this argument also says that the regularizing property (ii) together with the boundedness (i) imply $T(E_m) = T(E_0)\cap F_m$. Hence the inclusion $F_m\subset F_0$ induces an injection of cokernels $\qu{F_m}{T(E_m)} \hookrightarrow \qu{F_0}{T(E_0)}$, of which the latter is finite dimensional by (iii).
This proves that $T|_{E_m}$ also has finite dimensional cokernel, and hence is Fredholm as claimed.
In fact, since $F_m\subset F_0$ is dense, the image of the injection of cokernels must also be dense. But in finite dimensions that means equality, as claimed.
\end{proof}

\begin{lemma} \label{sc is hwz}
Let $\E, \F$ be sc-Banach spaces, then a linear map $T: E_0\to F_0$ is sc-Fredholm by Definition~\ref{fred op} iff it is sc-Fredholm by Definition~\ref{hwzfred}.
\end{lemma}
\begin{proof}
Given the splittings (iii) in Definition~\ref{hwzfred}, the finite dimensional summands are necessarily contained in the ``smooth'' intersection of all scales, $\ker T\subset E_\infty$ and $C\subset F_\infty$, since otherwise e.g.\ $C\cap F_\infty$ would be a proper subspace of $C$, which in finite dimensions contradicts the density axiom for sc-Banach spaces.
Next, any sc-Fredholm operator $T$ in the sense of by Definition~\ref{hwzfred} is regularizing by \cite[Prop.1.1.10]{hwzbook}, that is $T(E_0)\cap F_m = T(E_m)$. Together with $\ker T\subset E_\infty$ this indeed implies $T^{-1}\bigl(T(E_0)\cap F_m\bigr) \subset E_m$.
Moreover, each restriction $T|_{E_m}$ can be viewed as an operator between the direct sums $T|_{E_m} : \ker T \oplus (X\cap E_m) \to T(E_m) \oplus C$, where by (iii) the further restriction $T|_{X\cap E_m} : X\cap E_m \to T(E_m)$ is an isomorphism.
Since $\ker T$ and $C$ are finite dimensional, this implies that $T|_{E_m}$ is classically Fredholm, and hence $T$ is sc-Fredholm as per Definition~\ref{fred op}.

Conversely, given an sc-Fredholm operator $T$ as in Definition~\ref{fred op}, we have $\ker T \subset E_\infty$ by the regularizing property, and this kernel is finite dimensional by the Fredholm property of $T|_{E_0}$.
Then \cite[Prop.1.1.7]{hwzbook} provides an sc-complement $\E=\ker T \oplus_{sc} X$, that is $X\cap E_m$ is a topological complement for $\ker T \subset E_m$ for each $m\in\N_0$.
Next, $T(E_0)\subset F_0$ is closed and of finite codimension by the Fredholm property of $T|_{E_0}$, hence has a finite dimensional topological complement in $F_0$. Then \cite[Lemma 2.12]{hwz1} provides a finite dimensional subspace $C\subset F_\infty$ with $F_0= T(E_0) \oplus C$. We claim that this in fact yields an sc-direct sum
\begin{equation} \label{scale cokernel}
\F= T(E_0) \oplus_{sc} C .
\end{equation}
To check this we first ensure that $T(E_0)\cap F_m$ defines an sc-structure on $T(E_0)$. Indeed, by the regularizing property we have $T(E_0)\cap F_m=T(E_m)$, which is closed by Lemma~\ref{sc fred}, and hence inherits a Banach space structure from $F_m$.
Now the embedding $T(E_m)\hookrightarrow T(E_{m+1})$ is compact
since it is a restriction of the compact embedding $F_m\subset F_{m+1}$.
Moreover, $T(E_0)\cap F_\infty = T(E_\infty)$ is dense in every $T(E_m)$ since $E_\infty\subset E_m$ is dense and $T:E_m\to F_m$ is continuous.
Thus we have a scale structure on $T(E_0)$, along with the trivial scale structure $(C_m=C)_{m\in\N_0}$ on $C$. It remains to check that the direct sum isomorphism $\Pi_{T(E_0)}\times \Pi_C : F_0 \to  T(E_0) \times C$, given by continuous projection maps, is in fact an sc-isomorphism. On the finite dimensional space $C$ all norms are equivalent, so the continuity of $\Pi_C:F_0 \to C\subset F_0$ and $F_m\subset F_0$ implies continuity of
$\Pi_C:F_m \to C\subset F_m$,
$$
\bigl\| \Pi_C f \bigr\|_{F_m} \leq C_m \bigl\| \Pi_C f \bigr\|_{F_0} \leq C_m  C \| f \|_{F_0}
\leq C_m C C'_m \| f \|_{F_m} \qquad \forall f\in F_m .
$$
Now this implies continuity of $\Pi_{T(E_0)}|_{F_m}={\rm Id_{F_m}} - \Pi_C|_{F_m}$.
Hence we have established \eqref{scale cokernel}, that is, $C\subset F_m$ is a topological complement of $T(E_m)=T(E_0)\cap F_m$ for each $m\in\N_0$.
Finally, the restriction $T|_X : X \to T(E_0)$ is an sc-isomorphism since on every level $T : X\cap E_m \to T(E_m)$ is the restriction of a Fredholm operator to a map between the complement of the kernel and the image.
This proves that $T$ is also sc-Fredholm in the sense of Definition~\ref{hwzfred}.
\end{proof}

\section{Fredholm property for nonlinear maps} \label{s:nonlin}

The notion of a nonlinear Fredholm map on scale Banach spaces is not obtained by adding ``sc-'' in appropriate places to the classical definition of Fredholm maps, but requires tweaking to ensure an implicit function theorem for sc-Fredholm maps with surjective linearization. The latter is usually proven by means of a contraction property of the map in a suitable reduction. Since the contraction will be iterated to obtain convergence, it needs to act on a fixed Banach space rather than between different levels of a scale Banach space. 
In classical Fredholm theory, this contraction form follows from the continuity of the differential in the operator norm, whereas the differential of a scale smooth map is generally continuous only as operator between different levels. Hofer-Wysocki-Zehnder solve this issue by making the contraction property a part of the definition of Fredholm maps. However, this raises the question of how this property can be proven for a given map.
It turns out that in practice, this ``contraction germ normal form'' is established by proving classical continuous differentiability of the map in all but finitely many directions and the scale Fredholm property for this partial derivative.
We will formalize this approach in an alternative definition of the nonlinear Fredholm property, which is stronger than the following definition from \cite{hwzbook}, but is easier to check in practice.

Throughout we restrict our discussion to the Fredholm property at the zero vector $0\in\E$ in a scale Banach space. However, by a simple shift this provides the general Fredholm notion in polyfold theory. 

\begin{definition} [\cite{hwzbook} 3.1\footnote{
This definition is not explicitly given in \cite{hwzbook}. It is obtained from the definition of an M-polyfold Fredholm section of a strong bundle as the special case of the section $f(e)=(e,\Phi(e))$ in the trivial bundle $\E\times\F \to \E$ with trivial splicing (hence no filled section is involved). 
The reference sc$^+$ section $s:\cU \to \F$ will be given by $e \mapsto G\bigl( e , \Phi(0) \bigr)$, and is sc$^+$ due to the strong sc-smoothness of $G$ and $\Phi(0)\in F_\infty$. 
}] \label{def:HWZfred}
Let $\Phi:\E\to\F$ be a $sc^\infty$ map between sc-Banach spaces $\E, \F$.
Then $\Phi$ is {\bf scale Fredholm at $\mathbf{0}$} if the following holds:
\begin{enumlist}
\item
$\Phi$ is {\bf regularizing as germ}:
For every $m\in\N_0$ there exists $\epsilon_m>0$ such that
$\Phi(e)\in F_{m+1}$ and $\|e\|_{E_m}\leq \epsilon_m$ implies $e\in E_{m+1}$, 
\item
$\Phi$ has a {\bf contraction germ normal form} in sc-coordinates, that is:
\begin{itemize}
\item
There is an open sc-embedding $h:\cU \to \R^k\times\W$ (i.e.\ an sc$^\infty$ map to an open subset with sc$^\infty$ inverse) for some neighbourhood $\cU\subset E_0$ of $0$, some $k\in\N_0$, and some sc-Banach space $\W$, such that $h(0)=(0,0)$.
\item 
There is a germ of strong bundle isomorphism $G = \bigl(g_e:\F \to \R^\ell\times\W\bigr)_{e\in\cU}$ i.e.\ a family of linear bijections $g_e:F_0\to \R^\ell \times W_0$ for some $\ell\in\N_0$ 
so that for $i=0,1$ the map
$$
\qquad
G \,: \; \bigl( (E_m\cap \cU) \times F_{m+i} \bigr)_{m\in\N_0}  \;\to\;  \bigl( \R^\ell  \times W_{m+i} \bigr)_{m\in\N_0} , \quad
(e,f) \;\mapsto\; g_e(f)
$$ 
restricts to sc$^n$ maps on neighbourhoods $U_n\subset\cU$ of $0$ for every $n\in\N_0$.
\item
The transformed map is of the form
$$
G \circ \bigl( \Phi - \Phi(0) \bigr) \circ h^{-1} \,: \; (v,w) \; \mapsto\;  \bigl( A(v,w) , w - B(v,w) \bigr) ,
$$
where $A:\R^k\times\W \to \R^\ell$ is $sc^\infty$ and
$B:\R^k\times\W \to \W$ is a {\bf contraction germ}:
For every $m\in\N_0$ and $\theta>0$ there exists $\epsilon_{m,\theta}>0$ such that
for all $v\in\R^k$ and $w_1,w_2\in\W$
with $|v|_{\R^k}, \|w_1\|_{W_m}, \|w_2\|_{W_m} \leq \epsilon_{m,\theta}$ we have
\begin{equation}\label{contraction}
\bigl\| B(v,w_1) - B(v,w_2) \bigr\|_{W_m} \leq \theta \| w_1 - w_2 \|_{W_m} 
\end{equation}
\end{itemize}
\end{enumlist}
\end{definition}

In classical Fredholm theory, the above contraction germ normal form exists automatically for a continuously differentiable map whose differential at $0$ is Fredholm. The following remark explains this in detail and explores the failure of the analogous statement for scale smooth maps with sc-Fredholm differential.

\begin{remark}[Comparison with classical nonlinear Fredholm property] \label{rmk:contract}
\rm
Suppose that $\Phi:\E\to\F$ is a $sc^\infty$ map whose differential $\rD\Phi(0):\E\to\F$ is Fredholm, and let $\W\subset\E$ be a complement of its kernel. 
Then by Definition~\ref{def:sc}~(ii) of the differential we have $\Phi( h )  = \Phi ( 0 ) + \rD\Phi(0) h + R(h)$ with 
$\|R(h)\|_0 \le \eps({\|h\|_1}) \|h\|_1$ for a function $\eps:\R^+ \to \R^+$ with $\lim_{\hbar\to 0}\eps(\hbar)\to 0$.

Moreover, let $g=\iota^{-1}\oplus \rD\Phi(0)^{-1} :\F \to \R^\ell \times \W$ be the isomorphism induced by the direct sum $\F=\iota(\R^\ell)\oplus \rD\Phi(0)(\W)$ for a choice of complement $\iota:\R^\ell\hookrightarrow\F$ of the image of $ \rD\Phi(0)$.
Then, writing $h=v+w\in \ker\rD\Phi(0) \oplus\W$ we obtain
$$
g\bigl(\Phi( v+ w )\bigr)  \;=\; g\bigl(\Phi(0)\bigr) + \bigl( \iota^{-1}\bigl(\pr_{\iota(\R^\ell)}R(v+w)\bigr) , w - B(v,w) ,
\bigl)
$$
where $B(v,w):=- \rD\Phi(0)^{-1}\bigl(\pr_{\rD\Phi(0)\W}R(v+w)\bigr)$ is a contraction with respect to shifted scales from $W_1$ to $W_0$. That is, given $\theta>0$ we have
$$
\bigl\| B(v,w_1) - B(v,w_2) \bigr\|_{W_0} \leq \theta \| w_1 - w_2 \|_{W_1} 
$$
 for $\|v+w_1\|_{E_{1}}$, $\|w_1-w_2\|_{W_{1}}$ sufficiently small.
Indeed, $\Phi:E_1\to F_0$ is classically continuously differentiable by \cite[Prop.1.2.1]{hwzbook}, so that the mean value inequality for $v+w_1, v+w_2\in E_1$ gives for some $t\in[0,1]$
\begin{align*}
R(v+w_1) - R(v+w_2) 
&\;=\;
\Phi(v+w_1) - \Phi(v+w_2) + \rD\Phi(0) ( w_2 - w_1 )  \\
&\;=\;
\bigl( \rD\Phi\bigl(v+w_1+ t(w_2-w_1)\bigr) - \rD\Phi(0) \bigr)  ( w_1 - w_2  ) .
\end{align*} 
The continuity of the differential $\rD\Phi : E_1 \to L(E_1, F_0)$ then gives
$$
\bigl\|\rD\Phi\bigl(v+w_1+ t(w_2-w_1)\bigr) - \rD\Phi(0) \bigr\|_{L(E_1, F_0)} \leq \tfrac{\theta}{\| \rD\Phi(0)^{-1}\circ \pr_{\rD\Phi(0)\W} \|}
$$ 
for $\|v+w_1+ t(w_2-w_1)\|_{E_1}\le \|v+w_1\|_{E_1} + \| w_2-w_1\|_{W_1}$ sufficiently small.
This proves the above shifted contraction property, and if $\Phi:W_0\to W_0$ was classically continuously differentiable, then the same estimates would hold with $W_1$ replaced by $W_0$, thus establishing a contraction. In the scale differentiable case the analogous estimates can be established on all scales -- providing contractions from $W_{m+1}$ to $W_m$. However, this shift prevents iteration arguments such as the proof of Banach's fixed point theorem -- crucial part of the implicit function theorem.

Delving a little deeper into scale differentiability, \cite[Prop.1.2.1]{hwzbook} guarantees that $\Phi$ restricts to classically $\cC^1$ maps $E_{m+1}\to F_m$ for all $m\in\N_0$, and moreover its differentials extend to bounded linear maps in $L(E_m,F_m)$, which however depend continuously only on $E_{m+1}$ in the sense of continuity of the map 
\begin{equation}  \label{weak diff}
E_{m+1}\times E_m \to F_m,  \qquad(x,e)\mapsto \rD\Phi(x) e.
\end{equation}
%
%
Now in many applications the map $x\mapsto \rD\Phi(x)$ is not just continuous in the sense of \eqref{weak diff}, but in the operator topology as a map $E_{m+1} \to L(E_m,F_m)$. This implies a contraction estimate
$\bigl\| B(v,w_1) - B(v,w_2) \bigr\|_{W_m} 
\leq \theta \| w_1 - w_2 \|_{W_m} $
for $\|v\|_{W_{m+1}}$, $\|w_1\|_{W_{m+1}}$, $\|w_2\|_{W_{m+1}}$ sufficiently small.
However, the shift in norms still prevents the use of Banach's fixed point
theorem for $w=B(v,w)$ since closed $W_{m+1}$-balls are not complete in the $W_m$-norm.
\end{remark}

The above remark shows that an implicit function theorem for maps with surjective differential only follows from standard techniques if $\Phi:E_m\to F_m$ is $\cC^1$ in the standard sense. However, for Cauchy-Riemann operators, this stronger differentiability will fail as soon as $E_m$ contains gluing parameters which act on functions by reparametrization. This, however, is usually the only source of non-differentiability, and after splitting off a finite dimensional space of gluing parameters one deals with classical $\cC^1$-maps on all scale levels. 
If their differential would depend continuously on the gluing parameters in the operator topology, then the linear transformation of Remark~\ref{rmk:contract} would bring $\Phi$ into the contraction germ normal form that is required for $\Phi$ to be sc-Fredholm. In applications, this is generally not quite the case, but some weaker continuity still holds and suffices to find a nontrivial bundle isomorphism into a contraction germ normal form.
This motivates the following definition, which is just slightly stronger than the definition via contraction germs, but should be more intuitive for Cauchy-Riemann operators in the presence of gluing.
In fact, in practice the Fredholm property is proven via this stronger differentiability, see e.g.\ \cite[Prop.4.8]{hwz:gw}, and Remark~\ref{rmk:hwz} below.
Here we denote open balls centered at $0$ in a level $E_m$ of a sc-Banach space by
$$
B_r^{E_m} := \bigl\{ e \in E_m \,\big|\, \| e\|_m < r \bigr\} \qquad \text{for}\; r>0.
$$

\begin{definition} \label{def:scfred}
Let $\Phi:\E\to\F$ be a $sc^\infty$ map between sc-Banach spaces $\E, \F$.
Then $\Phi$ is {\bf sc-Fredholm at $\mathbf{0}$ with respect to the splitting $\E\cong\R^d\times\E'$} if the following holds.
\begin{enumlist}
\item
$\Phi$ is regularizing as germ, that is for every $m\in\N_0$ there exists $\epsilon_m>0$ such that
$\Phi(e)\in F_{m+1}$ and $\|e\|_{E_m}\leq \epsilon_m$ implies $e\in E_{m+1}$.
\item
$\E\cong\R^d\times\E'$ is an sc-isomorphism and for every $m\in\N_0$ there exists $\eps_m>0$ such that $\Phi(r,\cdot) : B_{\eps_m}^{E'_m} \to F_m$ is differentiable for all $|r|_{\R^d}<\eps_m$. 

Moreover, the differential $\rD_{\E'}\Phi(r_0,e_0) : \E' \to \F$, $e \mapsto \frac{\rm d}{{\rm d}t} \Phi(r_0,e_0+ te)|_{t=0}$ in direction of $\E'$ has the following continuity properties for any fixed $m\in\N_0$:
\begin{enumerate}
\item[a)]
For $r\in B_{\eps_m}^{\R^d}$ the differential operator $B_{\eps_m}^{E'_m} \to L(E'_m,F_m)$, $e \mapsto  \rD_{\E'}\Phi(r,e)$ is continuous, and the continuity is uniform in a neighbourhood of $(r,e)=(0,0)$. That is, for any $\d>0$ there exists $0<\eps_{m,\d}\le\eps_m$ such that for all $r \in B_{\eps_{m,\d}}^{\R^d}$ and $e,e'\in B_{\eps_{m,\d}}^{E'_m}$ we have
$$
\bigl\| \rD_{\E'}\Phi(r,e) h  - \rD_{\E'}\Phi(r,e') h \bigr\|_{F_m}  \le \d \| h \|_{E'_m} \qquad \forall  h \in E'_m .
$$
\item[b)]
For sequences $\R^d \ni r^\nu\to 0$ and $e^\nu\in B^{E'_m}_1$ with $\bigl\| \rD_{\E'}\Phi(r^\nu,0) e^\nu \bigr\|_{F_m} {\underset{\scriptscriptstyle \nu\to\infty}{\longrightarrow} 0}$ there exists a subsequence such that $\bigl\| \rD_{\E'}\Phi(0,0) e^\nu \bigr\|_{F_m} \underset{\scriptscriptstyle \nu\to\infty}{\longrightarrow} 0$.
\end{enumerate}
\item
The differential $\rD_{\E'}\Phi(0,0) : \E' \to \F$ is sc-Fredholm as in Definition~\ref{fred op}.
Moreover, $\rD_{\E'}\Phi(r,0) : E'_0 \to F_0$ is Fredholm for all $|r|_{\R^d} < \eps_0$, with Fredholm index equal to that for $r=0$, and these operators are weakly regularizing in the sense that for all $|r|_{\R^d} < \eps_0$ and $e\in E'_0$ we have the implication
$$
 \rD_{\E'}\Phi(r,0) e \in F_1  \qquad \Longrightarrow \qquad e \in E'_1.
$$ 
%
%
\end{enumlist}
\end{definition}

\begin{remark}\label{rmk:hwz} \rm
We can compare the above definition with the analytic properties of the Cauchy-Riemann operator in the Gromov-Witten case, from which \cite{hwz:gw} deduces its polyfold Fredholm property.
The regularizing property (i) in both Definitions~\ref{def:HWZfred} and \ref{def:scfred} is proven in \cite[Prop.4.17]{hwz:gw}.
Then conditions that imply the contraction germ form (ii) in Definition~\ref{def:HWZfred} are abstractly stated in \cite[Prop.4.26]{hwz:gw} and proven in \cite[Prop.4.23, 4.25]{hwz:gw}. 

The differentiability in all but finitely many directions (ii) in Definition~\ref{def:scfred} 
is the second bullet of \cite[Prop.4.23]{hwz:gw}, but not explicitly assumed in \cite[Prop.4.26]{hwz:gw}. 
The continuity in (ii-a) requires more local uniformity than \cite[Prop.4.25]{hwz:gw} and \cite[Prop.4.26~(3)]{hwz:gw}, which correspond to our case $e=0$. 
The sc-Fredholm property at $(0,0)$ in (iii) is also required by the first bullet of \cite[Prop.4.23]{hwz:gw} resp.\ \cite[Prop.4.26~(1)]{hwz:gw}, which moreover requires the sc-Fredholm property with the same index for all sufficiently small $(r,e)$. The latter is stronger than our classical Fredholm, index, and regularization conditions for small $(r,0)$, but follows from (iii) together with the continuous differentiability (ii) and the techniques of \S\ref{s:lin}.
Conversely, the regularization property (in fact, in the stronger version of (ii) in Definition~\ref{fred op}) holds in general for linear sc-Fredholm operators by Lemma~\ref{sc is hwz} or \cite[Prop.1.1.10]{hwzbook}.

It remains to compare the continuity in (ii-b) of Definition~\ref{def:scfred} with the third bullet of \cite[Prop.4.23]{hwz:gw} resp.\ \cite[Prop.4.26~(2)]{hwz:gw}. 
The special case of $K=\{0\}$, $(a^\nu,v^\nu)=r^\nu \to 0$ and $z^\nu=0$ in the latter is the assertion in our setting that a subsequence of $e^\nu$ converges in $E_m$. This then implies (ii-b) due to the continuity of $\R^d\times E'_m\to F_m$, $(r,e)\mapsto\rD_{\E'}\Phi(r,0)e$.
On the other hand, condition (ii-b) implies this special case of \cite[Prop.4.23]{hwz:gw} due to the estimate arising from injectivity of $\rD_{\E'}\Phi(0,0)$ on a complement of its finite dimensional kernel.
The general case of $\rD_{\E'}\Phi(r^\nu,k^\nu)- z^\nu \to 0$ for nontrivial $k^\nu\in K\subset \E'$ and $r^\nu\to r^\infty$, or $\|z^\nu\|_{F_{m+1}}\le 1$ is not a direct consequence of our conditions. However, a germ version for $(r^\nu,k^\nu)\to 0$ might follow from the contraction germ property.
\end{remark}

Using the weak continuity properties of the partial differential in all but finitely many directions, we can extend the techniques of Remark~\ref{rmk:contract} to obtain a contraction germ normal form for maps that are sc-Fredholm with respect to a splitting, and thus prove that they are essentially sc-Fredholm operators. Note here that in applications of the scale Fredholm theory, e.g.\ the implicit function theorem of polyfold theory, the contraction germ property is necessary only from some fixed scale onwards, so that Newton iteration can be performed on each sufficiently high scale to find a smooth solution set in the subset of ``smooth points'' of the polyfold.

\begin{theorem} \label{thm}
Let $\Phi:\E\to\F$ be a $sc^\infty$ map that is sc-Fredholm at $0$ with respect to a splitting $\E\cong\R^d\times\E'$ in the sense of Definition~\ref{def:scfred}. Then it satisfies all conditions of Definition~\ref{def:HWZfred} except for the contraction \eqref{contraction} for $m=0$.
In particular, $\Phi|_{E_1}:(E_{m})_{m\in\N}\to (F_{m})_{m\in\N}$ is sc-Fredholm at~$0$ w.r.t.\ the induced scale structures on $E_1\subset E_0$ and $F_1\subset F_0$. 
\end{theorem}

\begin{proof}
Since $\rD_{\E'}\Phi(0,0)$ is sc-Fredholm, Lemma~\ref{sc is hwz} provides sc-direct sums 
$$
\E' = \ker \rD_{\E'}\Phi(0,0) \oplus_{sc} \W , \qquad
\F= \im \rD_{\E'}\Phi(0,0) \oplus_{sc} C ,
$$
with $\W=(W_m:=W_0\cap E_m)_{m\in\N_0}$ and a finite dimensional subspace $C\subset F_\infty$.

Denote by $\Pi_C:\F\to C$ and $\Pi_C^\perp:=\id_\F-\Pi_C : \F \to \im \rD_{\E'}\Phi(0,0)$ the sc$^0$ projections to the factors, then we claim that for small $\eps'_0>0$ we obtain isomorphisms 
\begin{equation}\label{eq:iso}
\Pi_C^\perp \circ \rD_{\E'}\Phi(r,0)|_{W_0} \,:\;  W_0 \;\overset{\cong}{\longrightarrow}\;  \im \rD_{\E'}\Phi(0,0)  \qquad\forall |r|_{\R^d}\le \eps'_0 
\end{equation}
satisfying uniform estimates for all $m\ge 1$ with some $\eps'_m>0$, 
\begin{equation}\label{eq:est}
 \| w \|_{W_m}  \;\le\; C_m \bigl\| \Pi_C^\perp \, \rD_{\E'}\Phi(r,0)\, w \bigr\|_{F_m} \qquad \forall  |r|_{\R^d}\le \eps'_m , w \in W_m.
\end{equation}
Both hold by construction for $r=0$, so our claim is that, for a possibly larger constant $C_m$, they continues to hold for $ |r|_{\R^d}\le \eps'_m$ sufficiently small.
Note moreover that -- due to the finite codimensional restrictions in domain and target and the Fredholm condition (iii) on the map $\Phi$ -- we are dealing with sc-operators of Fredholm index 
\begin{align*}
&\ind \Pi_C^\perp\circ\rD_{\E'}\Phi(r,0)|_{W_0}
\;=\;
\ind \rD_{\E'}\Phi(r,0) - \dim\ker \rD_{\E'}\Phi(0,0) + \dim C \\
&\;=\;
\ind \rD_{\E'}\Phi(0,0) - \dim\ker \rD_{\E'}\Phi(0,0) + \dim \tfrac{F_0}{\im\rD_{\E'}\Phi(0,0)}
\;=\; 0.
\end{align*}
%
%
Hence for the isomorphism property \eqref{eq:iso} is suffices to prove injectivity of $\Pi_C^\perp\circ\rD_{\E'}\Phi(r,0)|_{W_0}$. This injectivity follows from the weak regularization property in Definition~\ref{def:scfred}~(iii), 
which for $w\in W_0$ guarantees
$$
\Pi_C^\perp\, \rD_{\E'}\Phi(r,0) w = 0 
\quad\Rightarrow\quad  \rD_{\E'}\Phi(r,0) w \in C \subset F_\infty 
\quad\Rightarrow\quad w\in E'_1 \cap W_0 , 
$$
and the estimate \eqref{eq:est} for $m=1$ on $E'_1 \cap W_0 = W_1$.
So it remains to prove \eqref{eq:est} for $r\ne 0$ and $m\geq 1$. For that purpose consider by contradiction sequences ${r^\nu\to 0}$, $\|w^\nu\|_{W_m} = 1$ for a fixed $m\in\N$ with $\|\Pi_C^\perp \, \rD_{\E'}\Phi(r^\nu,0)\, w^\nu \|_{F_m} \to 0$. 
This yields a uniform bound on
\begin{align*}
& \bigl\| \rD_{\E'}\Phi(r^\nu,0)\, w^\nu \bigr\|_{F_{m-1}} 
\;=\;  \bigl\| \rD \Phi(r^\nu,0)\, (0, w^\nu ) \bigr\|_{F_{m-1}} \\
&\quad \leq \bigl\| \rD\Phi(0,0) \bigr\|_{\cL(E_m,F_{m-1})}  + \bigl\| \rD \Phi(r^\nu,0) - \rD\Phi(0,0) \bigr\|_{\cL(E_m,F_{m-1})}
\end{align*}
since $\|(0,w^\nu)\|_{E_m}=\|w^\nu\|_{W_m}=1$.  
Indeed, the first summand is independent of $\nu$ and the second summand converges to zero as $\nu\to\infty$ since the $sc^1$-regularity of $\Phi$ in particular guarantees continuity of $ \rD\Phi : E_m \to \cL(E_m,F_{m-1})$ for $m\geq 1$; see \cite[Prop.1.2.1]{hwzbook}. 
This guarantees a bound on $\Pi_C \, \rD_{\E'}\Phi(r^\nu,0)\, w^\nu \in C$ in any norm on the finite dimensional subspace $C \subset F_\infty$, since $\Pi_C$ is a bounded map and all norms are equivalent in finite dimensions. Thus we can find a subsequence (again indexed by $\nu$) so that $\Pi_C \, \rD_{\E'}\Phi(r^\nu,0)\, w^\nu \to c_\infty \in C$ converges in $F_{m}$. Together with the contradiction assumption this implies $\bigl\| \rD_{\E'}\Phi(r^\nu,0)\, w^\nu - c_\infty \bigr\|_{F_m} \to 0$.

Moreover, since $m\geq 1$ so that $W_m\hookrightarrow W_{m-1}$ is compact, we find another subsequence so that $w^\nu\to w^\infty \in W_{m-1}$ converges in $W_{m-1}\subset E_{m-1}$.
This means that we have $(r^\nu,0,0,w^\nu) \to (0,0,0,w^\infty)$ converging on the $k=(m-1)$-th level of the scale tangent space $\rT(\R^d\times \E') = \bigl( \R^d \times E_{k+1} \times \R^d \times E_k\bigr)_{k\in\N_0}$, so that sc$^1$ regularity of $\Phi$ implies $F_{m-1}$-convergence $\rD_{\E'}\Phi(r^\nu,0)\, w^\nu \to \rD_{\E'}\Phi(0,0)\, w^\infty$.
In particular, since $F_m\subset F_{m-1}$ embeds continuously, we have $\rD_{\E'}\Phi(0,0)\, w^\infty = c_\infty \in C$, but since $w^\infty \in W_{m-1}$ by construction maps to the complement of $C$, this implies $c_\infty = 0$. 
So we have used the scale smoothness and injectivity of $\Pi_C^\perp \, \rD_{\E'}\Phi(0,0)$ to strengthen the assumption to $\bigl\| \rD_{\E'}\Phi(r^\nu,0)\, w^\nu \bigr\|_{F_m} \to 0$. At this point we can use the continuity property (ii-b) to deduce $\bigl\| \rD_{\E'}\Phi(0,0)\, w^\nu \bigr\|_{F_m} \to 0$ for a subsequence, so that finally \eqref{eq:est} for $r=0$ implies $\|w^\nu\|_{W_m}\to 0$ in contradiction to the assumption. 
This proves \eqref{eq:est} for $|r|_{\R^d}$ sufficiently small, in particular, it implies that $\Pi_C^\perp\,\rD_{\E'}\Phi(r,0)$ is an injective semi-Fredholm operator, and by the above index calculation, it is in fact an isomorphism which proves \eqref{eq:iso}.

Now an isomorphism $h:\E\cong\R^d\times\E' \to \bigl( \R^d \times \ker \rD_{\E'}\Phi(0,0) \bigr) \times \W$ of the base is given by splitting off the kernel of $\rD_{\E'}\Phi(0,0)$ from $\E'$ and adding it to the finite dimensional parameter space. So since the first factor $\R^d \times \ker \rD_{\E'}\Phi(0,0)$ is finite dimensional, we can equip it with the $E_0$-norm and find a bounded isomorphism to some $\R^k$. This is an sc-isomorphism since all $E_m$-norms restricted to the finite dimensional $\ker \rD_{\E'}\Phi(0,0)$ are equivalent.
Next, we obtain a bundle isomorphism $G=(g_e)_{e\in\cU}$ over $\cU:=\bigl\{ e\in E_0 \,\big|\, h(e) \in B^{\R^d}_{\eps'_0} \times  \ker \rD_{\E'}\Phi(0,0) \times \W\bigr\}$ by 
$$
G :  \cU \times \F \to  C \times \W , \;\;\;
(e,f) \mapsto \bigl( \Pi_C f \,,\, \bigl(\Pi_C^\perp \circ \rD_{\E'}\Phi( \pr_{\R^d}(h(e)),0)\bigr)^{-1} \Pi_C^\perp f  \bigr).
$$
Here the first factor $C$ is finite dimensional, hence sc-isomorphic to some $\R^\ell$.
Moreover, this map has the strong sc$^\infty$ regularity as germ near $\{0\}\times \F$ because
\begin{align*}
\bigl( B^{\R^d}_{\eps'_0} \times \ker \rD_{\E'}\Phi(0,0) \times W_m  \times F_{m+i} \bigr)_{m\in\N_0} & \;\to\;  \bigl( C \times W_{m+i} \bigr)_{m\in\N_0} , \\
(r,k,w, f) \;\mapsto\; &\bigl( \, \Pi_C f \,,\, \bigl(\Pi_C^\perp \circ \rD_{\E'}\Phi(r,0)\bigr)^{-1} \Pi_C^\perp f  \,\bigr)
\end{align*}
is independent of $k,w$, so that scale smoothness for $i\ge -m$ follows by the chain rule since $\Pi_C$, $\Pi_C^\perp$ are linear sc-operators and $(r,f)\mapsto \bigl(\Pi_C^\perp \circ \rD_{\E'}\Phi(r,0) \bigr)^{-1} f$  is a parametrized inverse to an sc$^\infty$ map, so that the usual formula for the derivative of an inverse proves scale smoothness.
These isomorphisms transform $\Phi$ to the map 
\begin{align*}
& \quad G\circ\bigl(\Phi-\Phi(0,0)\bigr)\circ h^{-1} : \; \R^d \times \ker \rD_{\E'}\Phi(0,0) \times \W \;\to\; C \times \W, \\
& (r,k,w) \;\mapsto\; g_{h^{-1}(r,k,w)} \bigl( \Phi(r,k+w) - \Phi(0,0) \bigr) =:  \bigl( A(r,k,w) , w - B(r,k,w) \bigr) .
\end{align*}
Here $A: \R^d \times \ker \rD_{\E'}\Phi(0,0) \times \W \to C$ and
$B: \R^d \times \ker \rD_{\E'}\Phi(0,0) \times \W \to \W$,
\begin{align*}
A(r,k,w) &:= \Pi_{C} \bigl( \Phi(r,k+w) - \Phi(0,0) \bigr)  ,\\
B(r,k,w) &:= w - \bigl( \Pi_C^\perp \circ \rD_{\E'}\Phi(r,0) \bigr)^{-1}\, \Pi_C^\perp \,\bigl( \Phi(r,k+w) - \Phi(0,0) \bigr) 
\end{align*}
are $sc^\infty$ by construction.
So it remains to establish the contraction germ property for $B$ for a fixed $m\in\N$ and $\theta>0$.  Note that we have
$$
-  \Pi_C^\perp \, \rD_{\E'}\Phi(r,0) \, B(r,k,w) = \Pi_C^\perp \,\bigl(
\Phi(r,k+w) - \Phi(0,0) - \rD_{\E'}\Phi(r,0) w \bigr),
$$
so that we can estimate, using the uniform bound \eqref{eq:est},  
\begin{align*}
& \bigl\| B(r,k, w_1) - B(r,k, w_2) \bigr\|_{W_m=W\cap E'_m} \\
&\qquad\quad\leq C_m \bigl\| \Pi_C^\perp \, \rD_{\E'}\Phi(r,0) B(r,k, w_1) - \Pi_C^\perp \,\rD_{\E'}\Phi(r,0) B(r,k, w_2) \bigr\|_{F_m} \\
&\qquad\quad= C_m \bigl\| \Pi_C^\perp\, \bigl(
\Phi(r ,k+w_1) - \Phi(r,k+w_2)
- \rD_{\E'}\Phi(r,0) (w_1-w_2) \bigr) \bigr\|_{F_m} \\
&\qquad\quad\leq C'_m \bigl\| \Phi(r,k+w_1) - \Phi(r,k+w_2)
- \rD_{\E'}\Phi(r,0) (w_1-w_2) \bigr\|_{F_m} 
\end{align*}
To prove the contraction germ property, let $\theta>0$ be given. Then we must bound the last expression by $\theta\|w_1-w_2\|_{W^m}$ for $\|(r,k)\|_{\R^d \times \ker \rD_{\E'}\Phi(0,0)}$, $\|w_1\|_{W^m}$, and $\|w_2\|_{W^m}$ sufficiently small.
For that purpose we will use the differentiability of $\Phi(r,\cdot) : B_{\eps_m}^{E'_m} \to F_m$ for $|r|_{\R^d}<\eps_m$.
By the triangle inequality, $\|k\|_{E'_m}$, $\|w_1\|_{E'_m}$, $\|w_2\|_{E'_m} < \frac 12\eps_m$ guarantees that $k+ \lambda w_1 + (1-\lambda) w_2 \in B_{\eps_m}^{E'_m}$ for all $\lambda\in[0,1]$.
Then $[0,1]\to F_m$, $\lambda \mapsto \Phi(r,k+ \lambda w_1 + (1-\lambda) w_2)$ is continuously differentiable, and hence we have 
\begin{align*}
&\Phi(r,k+w_1) - \Phi(r,k+w_2) - \rD_{\E'}\Phi(r,0) (w_1-w_2) \\
&\qquad\qquad= \int_0^1 \partial_\lambda \Phi\bigl(r,k+ \lambda w_1 + (1-\lambda) w_2\bigr)  \,{\rm d}\lambda  \;\; - \; \rD_{\E'}\Phi(r,0) (w_1-w_2)  \\
&\qquad\qquad= \int_0^1 \bigl( \rD_{\E'}\Phi(r,k+ \lambda w_1 + (1-\lambda) w_2)  - \rD_{\E'}\Phi(r,0) \bigr)  (w_1 - w_2) \,{\rm d}\lambda .
\end{align*}
%
%
Making use of the continuity properties of the differential in (ii-a) we may choose $0<\eps'_m\le \frac 12 \eps_m$ sufficiently small so that, with the given $\theta>0$ and constant $C_m$ from \eqref{eq:est}, we obtain
$$
\bigl\| \rD_{\E'}\Phi(r,k+w) - \rD_{\E'}\Phi(r,0)  \bigr\|_{L(E'_m,F_m)} \leq \tfrac{\theta}{C'_m} \quad
\forall |r|_{\R^d} , \|k\|_{E'_m} , \|w\|_{E'_m} \leq \eps'_m .
$$
Applying this estimate to $w= \lambda w_1 + (1-\lambda) w_2$ for $\|w_1\|_{E'_m}, \|w_2\|_{E'_m}\leq \eps'_m$ and $\l\in[0,1]$ then yields
$$
\bigl\| \Phi(r,k+w_1) - \Phi(r,k+w_2) - \rD_{\E'}\Phi(r,0) (w_1-w_2) \bigr\|_{F_m} \leq \tfrac {\theta}{C'_m} \|w_1-w_2\|_{E'_m}.
$$
Recalling that the $E'_0$ and $E'_m$ norms on $\ker \rD_{\E'}\Phi(r,0)$ are equivalent, we finally find $0<\d_m\le \eps'_m$ so that $\|(r,k)\|_{\R^d\times E'_0} < \d_m$ guarantees $|r|_{\R^d},  \|k\|_{E'_m}  \leq \eps'_m$. 
Now combining the above yields the contraction property: Given $m\in\N$ and $\theta>0$ we found $\d_m>0$ so that $\|(r,k)\|_{\R^d\times E'_0}$, $\|w_1\|_{E'_m}$, $\|w_2\|_{E'_m} < \d_m$ implies
\begin{align*}
& \bigl\| B(r,k, w_1) - B(r,k, w_2) \bigr\|_{W_m} \\
&\qquad\qquad\qquad \leq C'_m  \bigl\| \Phi(r,k+w_1) - \Phi(r,k+w_2) - \rD_{\E'}\Phi(0,0) (w_1-w_2) \bigr\|_{F_m}  \\
&\qquad\qquad\qquad \leq C'_m \cdot \tfrac {\theta}{C'_m} \|w_1-w_2\|_{E'_m} \;=\; \theta \cdot  \|w_1-w_2 \|_{W_m = W\cap E'_m} 
\end{align*}
This establishes the contraction germ property for $m\ge 1$ and hence shows that $\Phi|_{E_1}$ is sc-Fredholm at $0$ in the sense of Definition~\ref{def:HWZfred}.
\end{proof}

\begin{remark} \rm 
Conversely, in order to show that an sc$^\infty$ map $\Phi$ that is sc-Fredholm at $0$ is also sc-Fredholm at $0$ with respect to a splitting, the natural choice of splitting would be $\E\cong\R^k\times\W$, hoping that the contraction germ property for $B:\R^k\times\W\to\R^k$ implies differentiability in the direction of $\W$. 
However, this would require a stronger notion of contraction: For fixed $m\in\N_0$ we would need $\d_m>0$ independent of $\theta>0$, so that for any $\theta>0$ there exists $\epsilon_m>0$ so that for $\|v\|_{\R^k}$, $\|w_1\|_{W_m}$, $\|w_2\|_{W_m}\leq \delta_m$ we have contraction
\begin{align*}
\|w_1- w_2\|_{W^m} \leq \epsilon_m \; \Longrightarrow\; 
\bigl\| B(v,w_1) - B(v,w_2) \bigr\|_{W_m} \leq \theta \| w_1 - w_2 \|_{W_m} 
\end{align*}
This, however, would imply $\partial_{W_m} B \equiv 0$ in the $\delta_m$-neighbourhood.
So there seems to be no natural condition under which sc-Fredholm maps in the sense of Definition~\ref{def:HWZfred} are also sc-Fredholm with respect to a splitting as in Definition~\ref{def:scfred}.
\end{remark}

\section{Fredholm property in Hamiltonian Floer theory}
\label{s:Ham}

The full polyfold setup for Hamiltonian Floer theory will be a corollary of the polyfold setup for Symplectic Field Theory in \cite{fh-sft-full}. The purpose of this section is to demonstrate the polyfold Fredholm property of Floer's equation in the simplest setting that still captures the main subtleties, without going into the abstract notions of polyfold bundles.\footnote{
We use the appropriate language of polyfold theory but aim the exposition at non-experts, who may interpret unfamiliar terms intuitively or with the help of Appendix \ref{sec:polyfold}.}
For that purpose we restrict our attention to a neighbourhood of a once broken Floer trajectory in a simple geometric setting given at the beginning of \S\ref{ss:preglue}. We construct an ambient polyfold for this neighbourhood and a section with this zero set in \S\ref{ss:preglue} and \S\ref{ss:filling}, along with giving a rough idea of the abstract notions of polyfold and scale smooth section of a polyfold bundle. 
We then define the Fredholm property of the given section in terms of the sc-Fredholm property of a local ``filled section'' with the same zero set, which is rigorously stated in Theorem~\ref{mainthm} . (For a more rigorous introduction to the abstract polyfold notions see \cite{theguide}.) 
Finally, we give rigorous proofs of the less conventional Fredholm conditions (ii-a),(ii-b),(iii) of Definition~\ref{def:scfred} in \S\ref{ss:proof}. 

The adjustments needed for general once broken Floer trajectories will be indicated in remarks. 
Moreover, the preliminary section \S\ref{ss:fluff} -- which should be skipped by readers intent on rigorous content -- gives an overview of how the challenges of developing Floer theory in general symplectic manifolds (notably arising from sphere bubbling) are addressed by polyfold theory. In particular, we relate the classical gluing analysis to elements of the polyfold Fredholm notion, and explain why the estimates proven in \S\ref{ss:proof} are the key analytic ingredients for generalizing the polyfold setup for Gromov-Witten theory \cite{hwz:gw} to Floer theory. 

Readers interested in a polyfold setup for Morse theory can throughout replace ``(Hamiltonian) Floer'' by ``Morse'' and ``Cauchy-Riemann'' by ``gradient flow'', and instead of the first two paragraphs of \S\ref{ss:fluff} refer to e.g.\ \cite{schwarz:Morse} for an introduction. We will then indicate further simplifications in the series of remarks that also discusses the adjustments to general broken ``Floer-read-Morse'' trajectories.

\subsection{Historical context, polyfold overview, and sphere bubbling} \label{ss:fluff}

Hamiltonian Floer theory was introduced by Floer \cite{floer} in the 1980s to prove the Arnold conjecture \cite{floer:Arnold} for a class of symplectic manifolds in which energy concentration can be excluded. 
Textbook treatments in such settings are \cite{AD2010,Salamon,Sch}. 

During the 1990s, Floer theory was quickly gearneralized to all compact symplectic manifolds, where the main challenge was to regularize the pseudoholomorphic spheres which are the geometric result of energy concentration. 
However, due to the haste of the development some foundational issues were overlooked, as discussed in \cite{McDuffWehrheim}, and as of 2025 no general and generally understood proof seems to be published. 
Polyfold versions of the two basic proof approaches in \cite{floer} (deformation to the Morse complex) and \cite{pss} (direct isomorphism to homology) have been given in \cite{Zhengyi} and \cite{AFW} but rest on the pending polyfold Fredholm setup for SFT moduli spaces in \cite{fh-sft-full}.
While they, too, may not reach general understanding, the aim of these and the present paper is to demonstrate to non-experts how proofs based on regularization of moduli spaces can be made fully rigorous in the polyfold context while still following an intuitive line of argument.
As such we ought to include sphere bubbling in the discussion of the Fredholm property. 

Indeed, the compactifications of Floer's moduli spaces will generally contain multiply broken trajectories to which various trees of sphere bubbles are attached. Each such point in the moduli space is parametrized by a tuple of cylinders and spheres. (In the Morse case, cylinders are replaced by $\R$ and bubbles are absent.)
Moreover, any abstract regularization scheme such as polyfold theory must work with compact moduli spaces.\footnote{The basic example of abstract regularization (not to be confused with the regularizing property in Definitions~\ref{fred op} and \ref{def:scfred}) is the fact that the Euler class of a vector bundle is represented by the zero set of any transverse section -- if this zero set is compact. See \cite{McDuffWehrheim,theguide} for discussions of abstract and geometric regularization approaches.} Thus a full polyfold setup for Floer's moduli spaces requires a polyfold Fredholm section whose zero set is the compactified moduli space.
While this requires the construction of a section in a global polyfold bundle, its Fredholm property is local in the sense that we need to consider the section of a ``trivial polyfold bundle" over a polyfold chart, whose zero set is homeomorphic to an open neighbourhood in the compactified moduli space of a given multiply broken trajectory with sphere bubbles.

Prior to polyfold theory, these neighbourhoods would be described by a gluing theorem which combines classical Fredholm descriptions of the pieces -- Floer trajectories parametrized by cylinders and pseudoholomorphic spheres -- modulo reparametrizations of their domains. 
There are two analytic challenges in this which are also reflected in the polyfold approach, though the overall result is different. Classical gluing provides a homeomorphism from a fiber product of the individual moduli spaces, together with some gluing parameters, to an open subset of the uncompactified moduli space that is considered a neighbourhood of the broken/nodal trajectories that are being described.\footnote{If the individual moduli spaces are not cut out transversely, then the construction is still possible on the level of thickened, transversely cut out, moduli spaces. However, this requires choices of (spaces of) perturbations that achieve local transversality. While this local equivariant transversality can often be achieved, the dependence on choices yields a nontrivial notion of compatibility between charts -- the Kuranishi context, which has proven to be unexpectedly topologically treacherous in the later stages of regularization \cite{McDuffWehrheim}.}
The polyfold version of gluing provides a Fredholm section (in a new sense) whose zero set contains both the given broken/nodal trajectories and their neighbourhood of unbroken trajectories, thus cutting out an open subset of the compactified moduli space.

The first analytic challenge in both contexts is the lack of differentiability of the reparametrization action, which gets divided out in the moduli space (see \S\ref{s:calc} and \cite{theguide,McDuffWehrheim}). In the classical setting, this is resolved by finding perturbations of the Cauchy-Riemann operator\footnote{We will use ``Cauchy-Riemann operator" to refer to both the nonlinear differential operators $\CR_J$ given by the anti-holomorphic part of the differential of a map and its perturbation $\pbar$ by a Hamiltonian term in Floer's equation \eqref{eq:Floer}.
} 
that are both transverse and equivariant, so that one only has to take the quotient by this action on the solution space, where it is smooth.
In the polyfold setting, scale calculus allows to take the quotient on the general Banach space of maps. Then the Fredholm property of the Cauchy-Riemann operator has to be proven after restriction to a local slice of the action. Fortunately, sc-Fredholmness with respect to a trivial splitting in the sense of Definition~\ref{def:scfred} can be quickly deduced from well known properties of Cauchy-Riemann operators on appropriate Sobolev spaces of maps: 

\begin{enumlist}
\item
The regularization property follows from elliptic regularity. 
\item
The Cauchy-Riemann operator is not just differentiable with uniformly continuous derivatives, but in fact classically smooth w.r.t.\ appropriate Sobolev norms.
\item
Linearized Cauchy-Riemann operators are sc-Fredholm in the setting given after Definition~\ref{fred op}, i.e.\ they are 
(i) bounded between appropriate Sobolev spaces, (ii) regularizing by elliptic regularity, (iii) classically Fredholm.
All of these properties are preserved by restriction to a local slice; notably the Fredholm property because the local slice has finite codimension in the domain.
\end{enumlist}

The second analytic challenge is the fact that broken/nodal Floer trajectories are parametrized by different Riemann surfaces (unions of cylinders and spheres) than the nearby unbroken Floer trajectories, whose domain is a single cylinder. 
Classically, this is resolved by a pregluing construction which produces from a broken/nodal solution and choices of gluing parameters a map from a single cylinder, which solves Floer's equation to a small error on those parts of its domain where pregluing interpolates between different parts of the broken/nodal solution. Then a Newton iteration -- based on several subtle estimates -- is used to find a nearby exact solution. 
In the polyfold setting, the same analytic ingredients are used for different purposes:
The pregluing construction becomes a chart map which equips an ambient space of the compactified moduli space with a smooth structure; see \S\ref{ss:preglue}. 
The estimates going into the Newton iteration are part of the sc-Fredholm property, in particular ``quadratic estimates'' for varying the base point in linearizations of the Cauchy-Riemann operator, are very similar to those required as continuity of the differential in condition (ii) of the Definition~\ref{def:scfred}; see \S\ref{ss:proof}.

Rather than going deeper into the polyfold context, note that both pregluing and quadratic estimates are local considerations near a breaking or node. Thus we expect that a full polyfold setup -- just like a general gluing theorem -- near a complicated broken and nodal Floer trajectory can be pieced together from the setups for single breakings and nodes. That is, the polyfold Fredholm analysis is not just local in the moduli space but also domain-local in the sense that (most of) it can be performed by localizing to open subsets of the varying domains of the PDE. 
The analysis near a node is part of the polyfold Fredholm setup for Gromov-Witten theory in \cite{hwz:gw}.\footnote{ In particular, \cite{hwz:gw} describes trees of spheres attached to another pseudoholomorphic curve. In our context, the sphere trees are attached to a solution of the Floer equation, in which a zeroth order Hamiltonian term is added to the first order Cauchy-Riemann operator. This does not affect polyfold Fredholm properties, just as compact perturbations do not affect classical Fredholm properties.}
So it remains to discuss the polyfold setup for trajectory breaking at a fixed Hamiltonian orbit and the scale calculus of the Floer equation on a cylindrical end; and both can be done in a trivialized tubular neighbourhood of the orbit.

\subsection{Pregluing as polyfold chart map near broken trajectories} \label{ss:preglue}

We consider the simplified setting of a Morse function $H\in\cC^\infty(\C^n)$ 
which induces an autonomous Hamiltonian vector field $X:\C^n \to \C^n$
with nondegenerate (hence isolated) critical point $0\in{\rm Crit}(H)$. We moreover fix an almost complex structure $J\in\cC^\infty\bigl(\C^n, \Hom(\C^n,\C^n)\bigr)$, which induces a metric on $\C^n$ so that $X=J\nabla H$. 

A full polyfold Fredholm description of the (always compactified) Floer moduli space $\bM$ would be a Fredholm section $s:\cB\to\cE$ of a polyfold bundle $\cE\to\cB$ whose zero set $s^{-1}(0)\simeq\bM$ is homeomorphic to the moduli space. 
We localize our goal to describing a neighbourhood $s^{-1}(0)\simeq\cU\subset\bM$ in the moduli space near a once broken Floer trajectory from $c_0$ to $c_0$, broken at $c_0:S^1\to X, t\mapsto 0$. 
 
More precisely, a broken trajectory as stipulated is a pair $([\g_1],[\g_2])$ of solutions $\g_1,\g_2 \in\cC^\infty(\R\times S^1 ,\C^n)$ modulo reparametrizations\footnote{
Throughout, $[\cdot]$ denotes equivalence classes w.r.t.\ reparametrization by shifts in $\R$.
Since nontrivial Floer trajectories, by unique continuation, cannot be invariant under nontrivial $\R$-shifts, all our discussions will be in the realm of trivial isotropy. Thus all polyfolds will be special cases of trivial isotropy, also called M-polyfolds. 
The need for polyfolds in Floer theory arises only from the nontrivial isotropy of multiply-covered sphere bubbles. 
} 
of Floer's equation 
\begin{equation}\label{eq:Floer}
\pbar \g_i :=\partial_s \g_i + J(\g_i)\bigl( \partial_t \g_i - X(\g_i) \bigr) = 0 
\end{equation}
w.r.t.\ coordinates $(s,t)\in\R\times S^1$, with $\lim_{s\to\pm \infty}\g_i(s,\cdot)= c_0$ and finite energy
$$ \textstyle
\tfrac 12 \int_{\R\times S^1} | \partial_s \g_i |^2 + | \partial_t \g_i - X(\g_i) |^2 \;< \infty .
$$
An $\eps$-neighbourhood of the broken trajectory $([\g_1],[\g_2])$ in the Floer moduli space $\bM$ consists of those broken and unbroken trajectories that can be expressed as $([\g_1+\xi_1],[\g_2+\xi_2])$ or $[\oplus_R (\g_1+\xi_1, \g_2+\xi_2)]$ with $\|\xi_1\|, \|\xi_2\|<\eps$ and $R>e^{1/\eps}$.
The induced topology on the compactified moduli space is independent of the exact choice of norm $\|\cdot\|$, and so is mainly characterized by the {\bf pregluing} 
\begin{equation} \label{eq:preg}
\oplus_R(\xi_1,\xi_2) \,:=\; \b \cdot \tau_R \, \xi_1 + (1-\b) \cdot  \tau_{ - R} \, \xi_2 .
\end{equation}
Here $\beta\in\cC^\infty(\R,[0,1])$ is a cutoff function with $\beta|_{(\mi \infty,-1]}\equiv 1$ and $\beta|_{[1,\infty)}\equiv 0$ and  $\tau_R \, \xi (s,t) := \xi(R+s,t)$
denotes reparametrization by shifts in $\R$ direction.

\begin{remark}\rm  \label{rmk:gen1}
Morse theory also fixes a Morse function $H$ and a metric on $\C^n$. Then the gradient vector field $\nabla H$ replaces $-JX$ throughout, and we study $S^1$-invariant solutions $\g_i\in\cC^\infty(\R,\C^n)$ so that \eqref{eq:Floer} simplifies to $\partial_s \g_i + \nabla H(\g_i)=0$. Now we only assume $\lim_{s\to\infty}\g_1(s)=c_0=\lim_{s\to-\infty}\g_2(s)$ to allow for nonzero (but finite) energy $\int_\R \bigl| \partial_s \g_i \bigr|^2 = \lim_{s\to \infty} \g_i(s) - \lim_{s\to-\infty}\g_i(s)$. 

To describe the topology on the Floer moduli space near a general broken trajectory $([\g_+],[\g_-])$ of a periodic Hamiltonian on a symplectic manifold $M$, let $c_0:S^1\to M$ denote the common limit $c_0(t)=\lim_{s\to\pm\infty}\g_\pm(s,t)$.
We pick identifications $\g_\pm(\pm s,\cdot)^*\rT M \simeq c_0^*\rT M$ for large $s$ and a metric on $M$, which induces an exponential map.
Then for sufficiently small $\eps>0$ (which guarantees exponential maps), an $\eps$-neighbourhood of $([\g_+],[\g_-])$ consists of $([\exp_{\g_+}(\xi_+)],[\exp_{\g_-}(\xi_-)])$ and $[\exp_{\ti\g_R}(\oplus_R (\xi_+, \xi_-))]$ for small sections $\xi_\pm$ of $\g_\pm^*\rT M$ and large $R$. 
Here $\ti\g_R$ is an approximate concatenation of $\g_-$ and $\g_+$, which in the absence of a natural affine structure on $M$ is constructed as $\ti\g_R(s,t):=\ti\g_\pm(\pm R + s, t)$ for  $\pm s \leq 0$, where $\ti\g_+|_{(e^{1/\eps}-1, e^{1/\eps}) \times S^1}$ is a reparametrization of $\g_+|_{(e^{1/\eps}-1,\infty)\times S^1}$ and on the complement we set 
$\ti\g_+(s,t):=\g_+(s,t)$ for $s\le e^{1/\eps}-1$ and $\ti\g_+(s,t):=c_0(t)$ for $s\ge e^{1/\eps}$. 
Analogously, $\ti\g_-$ is obtained by reparametrizing $\g_-|_{(e^{1/\eps}-1,\infty)\times S^1}$ and setting $\ti\g_-(s,t):=c_0(t)$ for $s\le -e^{1/\eps}$. 
With that and the above identifications the pregluing \eqref{eq:preg} defines a section $\oplus_R(\xi_+,\xi_-)\in\ti\g_R^*\rT M$. 
\end{remark}

Now our first task is to construct an ambient space $\cB\subset\cU$ that contains a given $\eps$-neighbourhood $\cU$ of $([\g_1],[\g_2])$ and has some notion of smooth structure in which we can do Fredholm analysis. 
Here the core idea of polyfold theory is that pregluing ought to provide a chart map\footnote{
The smooth structure on this chart near the boundary $r=0$ depends on the choice of a homeomorphism $[0,\eps) \simeq (R_0,\infty]$. Our simplified choice $r\mapsto R(r):= e^{1/r}$ is equivalent to the ``exponential gluing profile'' $r\mapsto e^{1/r}-e$ used in \cite{hwz:gw}. 
} 
$$
(r,\xi_1,\xi_2) \;\mapsto\; 
\begin{cases}
\bigl[\oplus_R(\g_1+\xi_1,\g_2+\xi_2)\bigr] &; R=e^{1/r}<\infty, \\
\bigl( [\g_1+\xi_1],[\g_2+\xi_2]\bigr)  &; R=e^{1/r}=\infty,
\end{cases}
$$
which is a homeomorphism between an open subset of a smooth space
$$
\{ r,\|\xi_1\|,\|\xi_2\|<\eps  \} \subset [0,1) \times \cC^1(\R\times S^1,\C^n) \times \cC^1(\R\times S^1,\C^n)
$$
and a topological space $\cB$ that contains $\cU$ (with the induced relative topology identical to the topology given by pregluing).
Unfortunately, while this tentative chart map is continuous, open, and surjective by construction, it fails to be injective due to two separate effects.  
First, already $\xi_i\mapsto [\g_i+\xi_i]$ evidently maps orbits of the reparametrization $\R$-action to points. This can be remedied by taking a local slice in $\cC^1(\R\times S^1,\C^n)$ to the $\R$-action as in Remark~\ref{rmk:gen2}, which also rules out different shifts of $\xi_1,\xi_2$ and gluing lengths $R<\infty$ resulting in the same unbroken trajectory $[\oplus_R(\g_1+\xi_1,\g_2 + \xi_2)]$. 
However, the second effect is that the pregluing $\oplus_R(\xi_1,\xi_2)$ is independent of the values of $\xi_1|_{[R+1,\infty)\times S^1}$ and $\xi_2|_{(\mi \infty,-R-1]\times S^1}$.

To remedy this, polyfold theory uses open subsets of not just sc-Banach spaces but also sc-retracts (images of scale smooth retractions) as domains of chart maps. Trajectory breaking moreover introduces boundary, so that the smooth structure will arise from the product $[0, \eps ) \times \B$ of a closed interval (with natural sc-structure induced from $\R$) and an sc-Banach space $\B\subset \cC^1(\R\times S^1,\C^n) \times \cC^1(\R\times S^1,\C^n)$ which we construct in \S\ref{ss:sc}.  
Then the relevant scale smooth retraction\footnote{
The gluing profile $R(r)=e^{1/r}$ will be implicit in our notation throughout. This choice becomes crucial here since it guarantees scale smoothness of $\rho$.
}
\begin{align*}
\rho: \;  [0, \eps ) \times \B  &\;\to\;  [0, \eps ) \times \B \\
 (r,\xi_1,\xi_2)& \;\mapsto\; \bigl( r \, , \, \pi_r(\xi_1,\xi_2):= (\oplus_R\times\ominus_R)^{-1}( \oplus_R(\xi_1,\xi_2) , 0 ) \bigr) 
\end{align*}
arises from the family of projections $\pi_r:= (\oplus_R\times\ominus_R)^{-1}( \oplus_R(\cdot,\cdot) , 0 ) : \B\to\B$ along $\ker\oplus_R$ to the kernel of the {\bf anti-pregluing}
\begin{equation}
\label{eq:apreg}
\ominus_{ R}(\xi_1,\xi_2) \,:=\; (\b-1) \cdot   \tau_R\,\xi_1 +\b \cdot  \tau_{ - R} \,\xi_2 .
\end{equation}
This uses the (anti-)pregluing notation $\oplus_\infty(\xi_1,\xi_2):=(\xi_1,\xi_2)$ and $\ominus_\infty(\xi_1,\xi_2):=0$, so that $\oplus_R\times\ominus_R$ is an isomorphism for all $R\in [e^{1/\eps},\infty]$.
The resulting projections $\pi_r: \B\to\B$ are pointwise continuous w.r.t.\ variations of $r\in[0,\eps)$ but not continuous in any operator topology at $r=0$, thus allowing for jumping images from $\im\pi_0=\B$ to $\im\pi_r\subsetneq \B$ having infinite codimension. 
This provides a {\bf polyfold chart}\footnote{
For a rough idea of the compatibility notion between charts and a more detailed discussion of the retraction arising from pregluing see \cite[\S 2.1,\S 2.3]{theguide}.
} for an ambient space $\Ti\cB$ of the Floer moduli space $\bM$, whose image intersects $\bM$ in the $\eps$-neighbourhood $\cU$, 
\begin{align} 
\cB:=\{ (r,\xi_1,\xi_2)\in\im\rho \,|\, \|\xi_1\|,\|\xi_2\|<\eps  \} &\;\to\; \qquad \Ti\cB, \label{eq:chart} \\
(r,\xi_1,\xi_2)\qquad\qquad\quad &\;\mapsto\; [\oplus_R(\g_1+\xi_1,\g_2+\xi_2)] . \nonumber
\end{align} 
Since our discussion is localized to $\cU\subset\bM$, we can use $\cB:=\{ \|\xi_1\|,\|\xi_2\|<\eps  \}\subset\im\rho$ as ambient polyfold and move on to construct a section $s:\cB\to\cE$ whose zero set $s^{-1}(0)$ is identified with $\cU$ by the above chart map.

\begin{remark}\rm  \label{rmk:gen2}
In the general setting of Remark~\ref{rmk:gen1}, we need to replace addition in $\C^n$ by the exponential map in $M$, resulting in particular in the replacement of
$\oplus_R(\g_++\xi_+,\g_-+\xi_-) = \oplus_R(\g_+,\g_-) + \oplus_R(\xi_+,\xi_-)$ 
with $\exp_{\ti\g_R}\bigl( \oplus_R(\xi_+,\xi_-)\bigr)$.
Apart from that, the retraction is constructed as above since addition is well defined for the sections $\xi_\pm$ of $\g_\pm^*\rT M$ which will make up the sc-space $\B$ in this setting.  

Since the moduli space consists of Floer trajectories modulo $\R$-shifts, we also have to fix local slices. 
This can be done by finding codimension $1$ hypersurfaces $\Si_\pm \subset M$ that intersect $\g_\pm$ transversely at $(0,0)\in \R\times S^1$, and taking $\B$ to consist of sections $(\xi_+,\xi_-)$ that satisfy the slicing condition $\xi_\pm(0,0)\in \rT_{\g_\pm(0,0)}\Si_\pm$. 

The setup for Morse theory is obtained by dropping the variable $t\in S^1$ in the above constructions. Equivalently, we may take the $S^1$-invariant ``sub-polyfold''.
\end{remark}

\subsection{Floer's equation as polyfold bundle section, and its filling} \label{ss:filling}

The $\eps$-neighbourhood $\cU\subset\bM$ of the broken trajectory $\mathbf([\g_1],[\g_2])$ in the Floer moduli space is given by those equivalence classes of (pairs of) maps in the image of the chart map \eqref{eq:chart}, which satisfy Floer's equation \eqref{eq:Floer}.
Pulled back via the chart map, the solutions are $(r,\xi_1,\xi_2)\in\cB$ that satisfy 
$\pbar\oplus_R \!(\g_1+\xi_1,\g_2+\xi_2)=0$ resp. $\pbar (\g_i+\xi_i) = 0$ for $i=1,2$ in case $r=0$.  
To express this solution space as the zero set of a reasonable type of section
$$
s :\cB \to \bigcup_{(r, \xi_1,\xi_2)\in\cB}\{(r,\xi_1,\xi_2)\}\times \cE_{r} =: \cE, \quad (r,\xi_1,\xi_2) \mapsto \bigl( (r,\xi_1,\xi_2) , \phi(r,\xi_1,\xi_2) \bigr) 
$$
we must relate the fibers $\cE_r$ for different $r\ge 0$ (whch are independent of $(\xi_1,\xi_2)$). Naively, the fibers are $\cE_0\simeq \cC^0(\R\times S^1,\C^n) \times \cC^0(\R\times S^1,\C^n)$ for $r=0$ but $\cE_r\simeq \cC^0(\R\times S^1,\C^n)$ for $r>0$ since the first requires two PDEs on $\R\times S^1$ to be solved. 
This variation in fibers fortunately coincides with changes in the tangent spaces to the base polyfold, 
$\rT_{(r,\xi_1,\xi_2)}\cB \simeq \R \times \ker\ominus_R$, which are
$\rT_{(0,\xi_1,\xi_2)}\cB= \R\times \cC^1(\R\times S^1,\C^n)\times \cC^1(\R\times S^1,\C^n)$ for ${r=0}$ and for $r>0$ via $\id_\R\times\oplus_R$ are isomorphic to $\rT_{(r,\xi_1,\xi_2)}\cB \simeq \R\times \cC^1(\R\times S^1,\C^n)$.  
So it makes sense to build a {\bf polyfold bundle} $\cE\to\cB$ with fibers $\cE_{r} = \ker  \widehat\ominus_R \subset \E$ given by an anti-pregluing map $\widehat\ominus_R$ on another sc-Banach space $\E\subset\cC^\infty(\R\times S^1,\C^n) \times \cC^\infty(\R\times S^1,\C^n)$.\footnote{
The choices of $\B$ and $\E$ in \S\ref{ss:sc} will have to be compatible to make sure that the Floer operator $\pbar$ takes values in the fibers of $\cE$. So if $\B$ consists of $W^{k,p}$ Sobolev spaces, then $\E$ will use $W^{k-1,p}$, reflecting the fact that $\pbar$ is a first order differential operator. 
}
We can define pregluing $\widehat\oplus_R$ and anti-pregluing $\widehat\ominus_R$ on $\E$ by the same expressions as in \eqref{eq:preg}, \eqref{eq:apreg} to obtain an isomorphism $\widehat\oplus_R \times \widehat\ominus_R$ for each $r\in (0, \eps)$. 
Then we can rigorously define the {\bf polyfold bundle section} $s:\cB\to\cE$ above by the fiber part 
$$
\phi (r,\xi_1,\xi_2) :=\begin{cases}
(\widehat\oplus_R \times \widehat\ominus_R)^{-1} \bigl( \, \pbar \oplus_R \!(\g_1+\xi_1,\g_2+\xi_2)\, , 0 \, \bigr) &; r>0 \\
\bigl( \pbar (\g_1+\xi_1) , \pbar (\g_2+\xi_2) \bigr) &; r=0 .
\end{cases}
$$
Now the $\eps$-neighbourhood $\cU\subset\bM$ in the moduli space is identified, via the chart map \eqref{eq:chart}, with the zero set $s^{-1}(0):=\phi^{-1}(0)$ of the section $s$, given by the map
$$\textstyle
\phi: \; \bigcup_{r\in[0,\eps)} \{r\} \times \im\pi_r \supset \cB \;\to\;   \bigsqcup_{r\in[0,\eps)} \im\Hat\pi_r .
$$
Here $\Hat\pi_r:= (\Hat\oplus_R\times\Hat\ominus_R)^{-1}( \Hat\oplus_R(\cdot,\cdot) , 0 )$ is the family of projections arising from pregluing on $\E$. 
Now for appropriate sc-Banach spaces $\B,\E$ one can use the chain rule in scale calculus, together with the scale smoothness of reparametrization to check that $s$ is a {\bf scale smooth section}, in the sense that the induced map with open domain, $\phi\circ\rho:  [0,\eps)\times \B \supset \rho^{-1}(\cB) \to \E$ is sc$^\infty$.

Note however that $\phi\circ\rho$ cannot be Fredholm in any sense since $\rd\rho$ has infinite dimensional kernel on $(0,\eps)\times \B$, given by the kernel of pregluing. 
More abstractly, neither the base $\cB$ nor the bundle $\cE$ are locally homeomorphic to Banach spaces, but their tangent spaces $\R\times\im\pi_r$ resp.\ fibers $\im\Hat\pi_r$ are families of linear subspaces of the scale Banach spaces $\R\times\B$ resp.\ $\E$, which are parametrized by the same gluing parameter $r\in [0, \eps )$ such that ``base and fiber dimensions jump in the same way''. 
This is formalized by a {\bf filled section} $\Phi:\rho^{-1}(\cB)\to \E$ which has the same zero set $\Phi^{-1}(0)=\phi^{-1}(0)$ and restricts to isomorphisms $\{0\}\times\ker\pi_r \overset{\sim}\to \ker\Hat\pi_r$ between the complements of the fibers of $\rT\cB$ and $\cE$. 
Such a filling can often be achieved by acting on the anti-preglued map with a linearized operator -- in our setting the linearized Floer operator $\rD_{\overline c_0} \pbar$ at the constant trajectory $\overline c_0(s,t):=c_0(t)=0$. 
It gives rise to the scale smooth map $\Phi : [0, \eps) \times \B\supset \rho^{-1}(\cB) \to \E$ given by
replacing the $0$ in the definition of $\phi$ with $(\rD_{\overline c_0} \pbar) \ominus_R \!(\xi_1,\xi_2)$, i.e.\
\begin{equation}\label{Phi}
\Phi(r,\xi_1,\xi_2) :=
\begin{cases}
(\widehat\oplus_R \times \widehat\ominus_R)^{-1}
\left( \begin{aligned} 
\pbar \oplus_R \!(\g_1+\xi_1,\g_2+\xi_2) \\
(\rD_{\overline c_0} \pbar) \ominus_R \!(\xi_1,\xi_2) \quad
\end{aligned} \right) &;r>0, \\
\quad \left( \begin{aligned} 
 \pbar (\g_1+\xi_1) \\
  \pbar (\g_2+\xi_2) 
\end{aligned} \right) 
&; r=0 .
\end{cases}
\end{equation}

\begin{remark}\rm \label{rmk:gen3}
In the general setting of Remark~\ref{rmk:gen1}, these constructions are applied to spaces $\B,\E$ of sections $\xi_\pm$ of $\g_\pm^*\rT M$, after generalizing the formulas for $\phi,\Phi$ by
$\g_\pm+\xi_\pm = \exp_{\g_\pm}(\xi_\pm)$ and
$\oplus_R (\g_++\xi_+,\g_-+\xi_-) = \exp_{\ti\g_R}(\oplus_R (\xi_+,\xi_-))$. 

The Morse theory setting is again obtained by dropping the variable $t\in S^1$, and replacing the Floer operator $\pbar$ by $\partial_s + \nabla H$. This is equivalent to restricting the section $s:\cB\to\cE$ to the $S^1$-invariant subsets $\cB^{S^1}\subset \cB$ and $\cE^{S^1}\subset \cE|_{\cB^{S^1}}$.
 \end{remark}

\noindent
{\bf Summary:}
Given an $\eps$-neighbourhood $\cU\subset\bM$ of a broken trajectory $([\g_1],[\g_2])$, and up to the choice of $\B,\E$ in \S\ref{ss:sc}, the above constructs three types of maps whose zero sets are identified and homeomorphic to the given neighbourhood $s^{-1}(0)=\phi^{-1}(0)=\Phi^{-1}(0)\simeq\cU$ in the moduli space via \eqref{eq:chart}. 
Here $s:\cB\to\cE$ is a section of a polyfold bundle and $\phi:\im\rho\supset \cB\to\E$ is a scale smooth map from an open subset of an sc-retract to an sc-Banach space -- notions that are at best sketched here; see \cite{theguide} for more rigorous introductions.
However, $\Phi: [0,\infty)\times\B \supset \rho^{-1}(\cB)\to \E$ is a scale smooth map in the sense of Definition~\ref{def:sc}, if we allow for the generalization to maps whose domain are open subsets of sc-Banach spaces or -- the boundary case -- a product of a quadrant $[0,\infty)^k$ with an sc-Banach 
space.

\medskip 

Now the main result of this part is the Fredholm property of all these maps. For $s$ and $\phi$ this notion (from \cite[Def.3.6]{hwz2}, also see \cite[Def.6.2.8]{theguide}) requires for every ``smooth point''  $b\in\cB_\infty:= \im\rho\cap [0,\infty)\times B_\infty$ the existence of a filled section in local coordinates that is sc-Fredholm in the sense of Definition~\ref{def:HWZfred}. 
For the point $b=(0,0,0)\in\cB_\infty$, which corresponds to the broken trajectory $([\g_1],[\g_2])\in\bM$ in the chart \eqref{eq:chart}, such a filled section is given by $\Phi$, whose sc-Fredholm property at $b=0$ will be proven in Theorem~\ref{mainthm} below. 
At an unbroken map $b=[\g]\in\Ti\cB_\infty$ this germ Fredholm property follows -- in a simpler polyfold chart which requires no filling -- directly from well known properties of Cauchy-Riemann as outlined in \S\ref{ss:fluff}. 
At another once-broken map $b=([\g'_1],[\g'_2])\in\Ti\cB_\infty$, one would do the above setup with $\g_i$ replaced by the maps $\g'_i$. In general, these $\g'_i$ may not solve Floer's equation, but are smooth maps with exponential decay to the given limit orbits so that the proof of Theorem~\ref{mainthm}, in particular the estimates in \S\ref{ss:proof}, continue to hold.

More generally, as explained in \S\ref{ss:fluff}, the ambient polyfold $\Ti\cB_\infty$ of Floer's compactified moduli space $\bM$ will also contain multiply broken maps with bubble trees of spheres. 
To prove the Fredholm property at those points, the above setup and following Fredholm analysis will have to be ``patched together'' with the setup and analysis in \cite{hwz:gw} near nodes, as discussed in in \S\ref{ss:fluff}.

\begin{theorem}\label{mainthm}
The section $s:\cB\to\cE$ constructed above is polyfold Fredholm at $0$ in the sense that the filled section $\Phi$ is sc-Fredholm at $0$.
In fact, $\Phi$ is sc-Fredholm at $0$ with respect to the splitting $[0, \eps) \times \B$ of its domain, as in Definition~\ref{def:scfred}. 
\end{theorem}

\begin{proof}
The first part follows from the second part by definition and Theorem~\ref{thm}.
To see that $\Phi$ satisfies Definition~\ref{def:scfred}, we begin by noting that the regularization property (i) follows from standard elliptic regularity for the nonlinear and linearized Cauchy-Riemann operators together with the exponential decay property \eqref{eq:exp}.
\footnote{
For the regularization property it is crucial that the weight sequence $\ul\d$ is chosen between $0$ and the exponential decay constant $\d>0$ in \eqref{eq:exp}.
}
Next, $\Phi$ is composed of classically smooth maps (addition, multiplication with smooth scalar functions, and Cauchy-Riemann operators $\pbar, \rD_{\overline c_0} \pbar$ between appropriate Sobolev spaces) and reparametrizations $(r,\eta)\mapsto {\eta(\pm R+\cdot)}$ for the exponential gluing profile $R(r)=e^{1/r}$. The latter are scale smooth by combining exponential decay estimates with \cite[Thm.2.6]{hwz:retract}. Thus $\Phi$ is overall scale smooth by the chain rule in scale calculus \cite[Thm.2.16]{hwz1}. Moreover, $\Phi(r,\cdot)$ is classically smooth for fixed $r\in[0,\eps)$ by the classical chain rule since reparametrization $\tau_{\pmi R}$ with a fixed parameter $\pm R$ is a linear (hence smooth) operator.
So $\Phi$ satisfies the beginning of condition (ii) in Definition~\ref{def:scfred}, and it remains to establish (ii-a), (ii-b), (iii) in \S\ref{ss:proof}.
\end{proof}

\begin{remark}\rm \label{rmk:gen3.5}
Theorem~\ref{mainthm} extends to the general setting of Remark~\ref{rmk:gen3}. The necessary adjustments of arguments are discussed in Remark~\ref{rmk:gen3.9}. 

Similarly, Theorem~\ref{mainthm} holds in the Morse theory setting by the same proof.
Abstractly, the Morse section is obtained by restriction of the $S^1$-equivariant Floer section $s:\cB\to\cE$ to the fixed point set,  $s|_{\cB^{S^1}}:\cB^{S^1}\to (\cE|_{\cB^{S^1}})^{S^1}$.
Note here that both the ``sub-polyfold'' $\cB^{S^1}\subset \cB$ and the $S^1$-invariant subspace of each fiber $\cE_r^{S^1}\subset \cE_r$ have infinite codimension. While scale smoothness and regularization properties are preserved by this restriction, the specific Fredholm properties have to be deduced separately. Instead of the classical Fredholm properties of the Cauchy-Riemann operator, this is based on the Fredholm properties of the gradient flow operator $\xi\mapsto \partial_s(\exp_\g(\xi)) + \nabla H(\exp_{\g}(\xi))$. 
 \end{remark}

\subsection{Scale calculus for cylindrical ends} \label{ss:sc}

Recall from e.g.\ \cite{Salamon} that finite energy of the Floer trajectories $\g_i$ in \S\ref{ss:preglue} implies exponential decay $\lim_{s\to\pm \infty}\g_i(s, \cdot) = \g_{i,\pmi}$ to Hamiltonian orbits $\g_{i,\pmi}: S^1 \to \C^n$ with $\partial_t \g_{i,\pmi} (s,t) = X(\g_{i,\pmi})$, and we are considering the special case $\g_{i,\pmi}\equiv 0$.
Exponential decay for a smooth map $\g:\R\times S^1 \to \C^n$ in general means that for some constants $C$ and $\d>0$ we have
\begin{equation}\label{eq:exp}
 \bigl| \partial_s \g (s,t) \bigr| , \bigl| \partial_t \g (s,t) - X(\g_i(s,t)) \bigr| \le C e^{\mi\d |s|} \qquad\forall (s,t)\in\R\times S^1 
\end{equation}
and analogous estimates for all higher derivatives (which follow automatically if $\g$ satisfies the Floer equation \eqref{eq:Floer}).
Based on this decay constant $\delta>0$, we construct in Lemma~\ref{ex:sobolev} below (with $H^{m,\d}:=W^{m,2}_{\d}$) the ambient sc-Banach spaces
\begin{align}
\B &\,:=\; \bigl( B_m :=  H^{m+1,\delta_m}(\R\times S^1,\C^n) \times H^{m+1,\delta_m}(\R\times S^1,\C^n) \bigr)_{m\in\N}   \label{eq:BE}\\
\E &\,:=\; \bigl( E_m :=  H^{m,\delta_m}(\R\times S^1,\C^n) \times H^{m,\delta_m}(\R\times S^1,\C^n) \bigr)_{m\in\N}  \nonumber
\end{align}
For $\Phi:[0,\eps)\times \B\to\E$ given by \eqref{Phi} to be sc-Fredholm between these spaces, it is crucial to use the same weight sequence $0\le \d_1< \d_2 <\ldots<\d$ and a regularity shift by $1$ between $\B$ and $\E$. This ensures that $\Phi$  is scale smooth and regularizing, and in this setting \S\ref{ss:proof} proves the sc-Fredholm property as stated in Theorem~\ref{mainthm}.

\begin{remark}\rm \label{rmk:gen2.5}
Note that we dropped from $\B$ the slicing conditions of Remark~\ref{rmk:gen2}, since they do not affect the analysis in \S\ref{ss:proof}.
Remark~\ref{rmk:gen4} provides $\B,\E$ that yield a rigorous polyfold Fredholm setup in \cite{fh-sft}. This requires two restrictions: first to higher regularity $(B_m)_{m\ge 2}, (E_m)_{m\ge 2}$ and second to codimension $1$ subspaces.
So this setup considers Floer's equation as operator $H^3\to H^2$, whereas our setup works with ambient spaces $H^2\to H^1$. 
On the other hand, viewing Floer's equation as operator $H^1\to H^0$ corresponds to extended sc-Banach spaces 
\begin{equation} \label{eq:BEext}
\B^0 \,:=\; \bigr( H^{m+1,\d_m} \times H^{m+1,\d_m} \bigl)_{m\in\N_0} , \quad \E^0 \,:=\; (  H^{m,\d_m}\times H^{m,\d_m} )_{m\in\N_0} .
\end{equation}
For maximal generality we allow weight sequences $-\d< \d_0 <\d_1<\ldots<\d$. Then any choice of $\B,\E$ in \eqref{eq:BE} can be viewed as restriction of \eqref{eq:BEext} to $k\in\N$ for some choice of $-\d<\d_0<\d_1$. 
More precisely, we drop the restriction $\d_0\ge 0$ only when considering linearized operators. In the setup of the nonlinear map $\Phi$ in \eqref{Phi} the condition $\xi_i \in H^{1,\d_0}\subset H^1$ is needed to guarantee $\lim_{s\to\pm\infty} (\g_i+\xi_i)(s,\cdot)=0$. 

In \S\ref{ss:proof} we establish most sc-Fredholm properties of $\Phi:[0,\eps)\times \B^0\to\E^0$, but we cannot establish the full sc-Fredholm property due to nonlinearities in the proof of Lemma~\ref{le1}, analogous to the ``quadratic estimates'' in classical gluing analysis. 
\end{remark}

To fix notation and recall the scale calculus from \S\ref{s:calc}, we construct the weighted Sobolev spaces and organize them into sc-Banach spaces. 
Note that we used $p=2$ above since perturbation results in \cite{hwz2} ultimately require sc-Hilbert spaces.

\begin{lemma} \label{ex:sobolev}
The weighted Sobolev space with scale structure 
$$
\W^{\ell,p}_{\ud}(\R\times S^1,\C^n) \;=\; \bigl( W^{\ell+m,p}_{\delta_m}(\R\times S^1,\C^n) \bigr)_{m\in\N_0}
$$
is an sc-Banach space for any $n\in\N$, $\ell\in\N_0$, $1\leq p <\infty$, and weight sequence $\ud=(\delta_m)_{m\in\N_0}$ with $\delta_{m+1}>\delta_m\ge0$.
It is defined by the weighted Sobolev spaces
$$
W^{k,p}_{\delta}(\R\times S^1,\C^n) 
:= \bigl\{ u:\R\times S^1\to\C^n \st (s,t)\mapsto e^{\delta \eta(s)} u(s,t) \in W^{k,p} \bigr\}
$$
with norm $\| u \|_{W^{k,p}_{\delta}} := \| e^{\delta \eta} u \|_{W^{k,p}}$ for some choice\footnote{
Different choices of $\eta$ yield equivalent norms on the same space, which we use interchangeably.
Another equivalent definition is $u\in W^{k,p}_\delta$ iff $e^{\delta |s|} \nabla^\alpha u \in L^p$ for all multi-indices $|\alpha|\leq k$. The corresponding equivalent norm $\sum_{|\alpha|\leq k} \|e^{\delta |s|} \nabla^\alpha u\|_{L^p}$ is often used in proving estimates.
} 
of $\eta\in\cC^\infty(\R)$, with $\eta(s)=|s|$ for $|s|\ge 1$ and $0<\eta(s)<1$ for $|s|<1$.
\end{lemma}
\begin{proof}
The inclusion $E_k=W^{k+\ell,p}_{\delta_k}(\R\times S^1,\C^n) \subset W^{m+\ell,p}_{\delta_m}(\R\times S^1,\C^n)=E_m$ for $k>m$ exists since 
$e^{\delta_k \eta} \geq e^{\delta_m \eta}$. It is compact since the restriction
$W^{k+\ell,p}_{\delta_k}(\R\times S^1,\C^n) \to W^{m+\ell,p}_{\delta_k}([\mi R,R]\times S^1,\C^n)$ is a compact Sobolev embedding for any finite $R\geq 1$ (due to the loss of derivatives $k>m$, see \cite{Adams}) and the restriction
$W^{k+\ell,p}_{\delta_k}(\R\times S^1,\C^n) \to W^{k+\ell,p}_{\delta_m}((\R\setminus[\mi R,R])\times S^1,\C^n)$ converges to $0$ in the operator norm as $R\to\infty$ (due to the exponential weights $\delta_k>\delta_m$ combining to $\sup_{|s|\geq R} e^{\delta_m \eta(s)} e^{\mi \delta_k \eta(s)} = e^{\mi(\delta_k-\delta_m)R}$).

The {\em smooth points} $u\in E_\infty:= \bigcap_{m\in\N_0}W^{\ell+m,p}_{\delta_m}(\R\times S^1,\C^n)$ are those smooth maps $u\in\cC^\infty(\R\times S^1,\C^n)$ whose derivatives decay exponentially in the sense that 
$\sup_{s,t\in\R\times S^1} e^{\delta' \eta(s)} | \partial_s^{N_1}\partial_t^{N_2} u (s,t) | < \infty$ for all $N_1,N_2\in\N_0$ and any submaximal weight $\delta' < \sup_{m\in\N_0} \delta_m$.
In particular, the compactly supported smooth functions are a subset 
$\cC^\infty_0(\R\times S^1,\C^n)\subset E_\infty$; and even these are dense 
in any weighted Sobolev space (for $p<\infty$).
\end{proof}

\begin{remark}\rm \label{rmk:gen4}
For the slicing conditions of Remark~\ref{rmk:gen2} to define a local slice to the reparametrization in $\R$, the space $\B$ has to embed into continuously differentiable functions. 
This can be achieved by starting the scales at a weighted $H^3$ space. 
Moreover, the general setting of Remark~\ref{rmk:gen1} requires sections of pullback bundles. So,  abbreviating $H^{k,\d} :=  H^{k,\delta}(\R\times S^1,\g_+^*\rT M) \times H^{k,\delta}(\R\times S^1,\g_-^*\rT M)$, the general sc-Banach spaces $\B=(B_m)_{m\in\N_0}$ and $\E=(E_m)_{m\in\N_0}$ are given by
\begin{align*}
B_m &\,:=\; \{ (\xi_+,\xi_-)\in H^{m+3,\d_m} \st  \xi_\pm(0,0)\in \rT_{\g_\pm(0,0)}\Si_\pm\}, \quad
E_m \,:=\; H^{m+2,\delta_m}
 \end{align*}
for $0\le\d_0<\d_1<\ldots<\d$ and hypersurfaces $\Si_\pm\subset M$ as in Remark~\ref{rmk:gen2}. 
Here $\d>0$ is the maximal constant for which both the exponential decay estimates \eqref{eq:exp} for Floer trajectories and invertibility of the linearized operators in Lemma~\ref{lem:invert} hold. Explicitly, $\d>0$ is the smallest absolute value of eigenvalues of the limit operators $J(c)\partial_t - \rD_{c}(JX)$ on $L^2(S^1,c^*\rT M)$ at Hamiltonian orbits ${c:S^1\to M}$. 

In the Morse setting of Remark~\ref{rmk:gen3.5}, this limit operator of the Hessian of the Morse function $H$ at critical points $c\in {\rm Crit}(H)$. 
The $S^1$-invariant subspaces $\B^{S^1}\subset\B$, $\E^{S^1}\subset\E$ also form sc-Banach spaces by restriction. Alternatively, one can directly check the embedding properties for the sc-Banach spaces given by
\begin{align*}
B^{S^1}_m &:= \bigl\{ \bigl(\xi_\pm \in H^{m+2,\delta_m}(\R,\g_\pm^*\rT M)\bigr)\st  \xi_\pm(0)\in \rT_{\g_\pm(0)}\Si_\pm \bigr\}, \\
E^{S^1}_m &:=  H^{m+1,\delta_m}(\R,\g_+^*\rT M) \times H^{m+1,\delta_m}(\R,\g_-^*\rT M)  .
\end{align*}
The only difference in the use of Sobolev embeddings on $1$-dimensional domains is that $H^{k}(\R)\subset\cC^1(\R)$ holds for $k\ge 2$, so we can extend the scales to view the gradient flow as operator $H^2\to H^1$. However, the perturbation scheme for Morse theory that arises from the above setup starting at $H^3\to H^2$ would be the same.
\end{remark}

This completes the construction of the polyfold bundle section $s:\cB\to\cE$ and its filled section $\Phi: [0,\infty)\times\B \supset \rho^{-1}(\cB)\to \E$ in \S\ref{ss:filling}. 
Now the precise claim of Theorem~\ref{mainthm} is that $\Phi$ is sc-Fredholm at $0$ with respect to the splitting $\R^d\times\E'=\R\times\B$ (and with target $\F=\E$, not to be confused with the domain space that was denoted $\E$ in \S\ref{s:nonlin}).
That is, $[0, \eps )$ plays the role of the finite dimensional parameter space $\R^d$ and conditions (ii), (iii) require an understanding of the partial differential $\rD_\B\Phi(r,\ul e)$ at $r\in [0, \eps )$ and $\ul e = (e_1,e_2)\in\E'=\B$.
To calculate it we abbreviate 
$$
\g_r := \oplus_R(\g_1, \g_2),\;\;
\g^e_r := \oplus_R(\g_1 + e_1, \g_2 + e_2), \;\;
\g_r^\pmi = \tau_{\pmi R}\g_r ,\;\;
\g_r^{e\,\pmi} = \tau_{\pmi R}\g^e_r 
$$
for $r>0$ and for $r=0$ set $\g^\mi_0 := \g_1$, $\g^\pl_0 := \g_2$, 
$\g^{e\,\mi}_0 := \g_1+e_1$, $\g^{e\,\pl}_0 := \g_2+e_2$, 
which coincides with the pointwise $r\to 0$ resp.\ $R\to\infty$ limits of $\g_r^{e\,\pmi}=\tau_{\pmi R} \oplus_R(\g_1+e_1, \g_2+e_2)$. 
With that notation, $\rD_\B\Phi(r,\ul e) (\xi_1,\xi_2)$  for $r>0$ is given by
\begin{align}  
& \rD_\B\Phi(r,\ul e) (\xi_1,\xi_2) \nonumber \\
&\;=\;  \left.\frac\rd{\rd h}\right|_{h=0} 
(\widehat\oplus_R \times \widehat\ominus_R)^{-1} 
\left( \begin{aligned}
\pbar \oplus_R \!(\g_1+e_1+h\xi_1,\g_2+e_2+h\xi_2)   \\
(\rD_{\overline c_0} \pbar) \ominus_R \!(e_1+h\xi_1,e_2+h\xi_2) 
\end{aligned}
\right) \label{defDPhi}  \\
&\;=\;
(\widehat\oplus_R \times \widehat\ominus_R)^{-1} 
\left( \begin{aligned}
\bigl( \partial_s  + J(\g^e_r) \, \partial_t + F_{\g^e_r} \bigr) \oplus_R \! (\xi_1,\xi_2)  
\\
 \bigl( \partial_s  + J(\overline c_0) \partial_t  + F_{\overline c_0}   \bigr) \ominus_R \! (\xi_1,\xi_2)  
\end{aligned}
\right) \nonumber.
\end{align}
Here we encounter linearized Floer operators (w.r.t.\ implicit connections) at general base points $\g\in\cC^\infty(\R\times S^1,\C^n)$ with limits $\lim_{s\to\pm\infty}\g(s,\cdot)=c_0$ as in \eqref{eq:exp}. They are well defined for any $m\in\N_0$ and weight $w\in\R$ 
by
$$
\rD_\g \pbar = \partial_s  + J(\g) \, \partial_t + F_{\g} :  
H^{m+1,w}(\R\times S^1,\C^n) \;\to\; H^{m,w}(\R\times S^1,\C^n), 
$$
with the pointwise linear operators $F_\g: \C^n\to \C^n$ given by\footnote{
The Morse case replaces $\rD_\g \pbar$ by  $\partial_s  + F_{\g}$ with the Hessian $F_\g(\zeta)= \rD_\g \nabla H (\zeta)$.
}
\begin{equation}\label{eq:F}
F_\g \,:\; \zeta \;\mapsto\; (\rD_\g J)(\zeta) \partial_t \g - \rD_\g (JX) \zeta .
\end{equation}
For $r > 0$ we then use the linear algebra formula
\begin{equation}\label{eq:totalglue}
(\widehat\oplus_R \times \widehat\ominus_R)^{-1} (\zeta_1,\zeta_2) 
\;=\;
\left( \begin{aligned} 
\tau_{\mi R}  \bigl( \tfrac{\b}{\b^2 + (1-\b)^2} \zeta_1 + \tfrac{\b-1}{\b^2 + (1-\b)^2} \zeta_2 \bigr)  \\
\tau_{R}  \bigl(\tfrac{1-\b}{\b^2 + (1-\b)^2} \zeta_1 + \tfrac{\b}{\b^2 + (1-\b)^2}  \zeta_2 \bigr) 
\end{aligned}
\,\right) 
\end{equation}
to obtain
\begin{align} \label{DPhi}  
\rD_\B\Phi(r,\ul e) (\xi_1,\xi_2)
\;=\;
\left( \begin{aligned}
(\rD_{\g^\mi_r} \pbar ) \, \xi_1 \;+\; E_1(r,\ul{e},\xi_1,\xi_2) \\
(\rD_{\g^\pl_r} \pbar ) \, \xi_2 \;+\; E_2(r,\ul{e},\xi_1,\xi_2)
\end{aligned}
\, 
\right) .
\end{align}
Here the error terms $E_i$ for $r>0$ in terms of cutoff functions $B^{\cdots}$ (whose superscripts indicate their supports) are\footnote{
In the Morse setting, the error term $E_1(r,\ul{e},\xi_1,\xi_2)$ simplifies slightly to
\begin{align*}
&\quad B^{\scriptscriptstyle (\mi \infty,R+1]}  \bigl( \rD_{\g^{e\,\mi }_r}\nabla H - \rD_{\g^{\mi }_r}\nabla H \bigr) \,\xi_1 
+
B^{\scriptscriptstyle [R-1,\infty)}  \bigl( \rD_{\overline c_0}\nabla H - \rD_{\g^{\mi}_r}\nabla H \bigr) \,\xi_1 \\
&
+ B^{\scriptscriptstyle [R-1,R+1]}  \bigl( \rD_{\g^{e\,\mi}_r}\nabla H - \rD_{\overline c_0}\nabla H \bigr) \tau_{\mi 2R}\,\xi_2
+ B_s^{\scriptscriptstyle [R-1,R+1]} \bigl(\tau_{\mi R}(2\b - 1) \cdot  \xi_1  - \tau_{\mi 2R}\, \xi_2 \bigr).
\end{align*}
}
\begin{align*}  
& E_1(r,\ul{e},\xi_1,\xi_2)  \\
&\;\;= \;\; B^{(\mi\infty,R+1]}  \bigl( \bigl(J(\g^{e\,\mi}_r)-J(\g^\mi_r) \bigr) \, \partial_t  \xi_1 
\;+\;  \bigl( F_{\g^{e\,\mi}_r} - F_{\g^{\mi}_r} \bigr) \,\xi_1  \bigr) \\
&\qquad
\;+\;
B^{[R-1,\infty)} \bigl( \bigl(J(\overline c_0)-J(\g^{\mi}_r) \bigr) \, \partial_t  \xi_1 
\;+\;  \bigl( F_{\overline c_0} - F_{\g^{\mi}_r} \bigr) \,\xi_1 \bigr) \\
&\qquad
\;+\;  B^{[R-1,R+1]} \bigl( \bigl(J(\g^{e\,\mi}_r)-J(\overline c_0) \bigr)\, \tau_{\mi 2R}\,\partial_t \xi_2 
\;+\;\bigl( F_{\g^{e\,\mi}_r} - F_{\overline c_0} \bigr)\,  \tau_{\mi 2R}\,\xi_2 \bigr) \\
&\qquad
\;+\;
 B_s^{[R-1,R+1]} \bigl( \tau_{\mi R}(2\b - 1) \cdot \xi_1  
\;-\;  \tau_{\mi 2R}\, \xi_2 \bigr) ,
\end{align*}
and analogously for $E_2(r,\ul{e},\xi_1,\xi_2)$ by ``swapping cylinders'' (this involves $\g^{e\,\pl}_r$, $\g^\pl_r$, $\overline c_0$ and $\xi_2,\tau_{2R}\,\xi_1$). The precise cutoff functions are
\begin{align*}
B^{(\mi \infty,R+1]} :=  \tau_{\mi R} \Bigl( \tfrac{\b^2}{\b^2 + (1-\b)^2} \Bigr) , \quad
&B^{[R-1,\infty)} :=  \tau_{\mi R} \Bigl( \tfrac{(1-\b)^2}{\b^2 + (1-\b)^2} \Bigr) , \phantom{\int_{B_B}} \\
 B^{[\pmi R-1,\pmi R+1]} :=  \tau_{\mip R} \Bigl( \tfrac{\b(1-\b)}{\b^2 + (1-\b)^2} \Bigr) , \quad
&B^{[\pmi R-1, \pmi R+1]}_s :=    \tau_{\mip R}\Bigl(\tfrac{\partial_s \b}{\b^2 + (1-\b)^2}\Bigr) .
\end{align*}
Next, note that \eqref{DPhi} continues to hold for $r=0$ with error term
\begin{align*}  
E_i(0,\ul{e},\xi_1,\xi_2) 
&= \bigl(J(\g_i + e_i)-J(\g_i) \bigr) \, \partial_t  \xi_i 
+ \bigl( F_{\g_i+e_i} - F_{\g_i} \bigr) \,\xi_i  ,
\end{align*}
which coincides with the previous formulas for $E_i(r,\ul{e},\xi_1,\xi_2)$ if for $R=\infty$ we set
${B^{(\mi\infty,R+1]}\equiv 1}$, $B^{[R-1,\infty)}\equiv 0$, $B^{[R-1,R+1]}\equiv 0$, $B_s^{[R-1,R+1]}\equiv 0$.
Moreover, note that $\ul e=0$ yields $E_i(0,0,\xi_1,\xi_2)=0$, and for $r>0$ we can split this error 
$$
E_i(r,0,\xi_1,\xi_2) = E'_i(r,\xi_1,\xi_2) + B_i(r,\xi_1,\xi_2) 
$$
into terms $E'_i$ that are controlled by $\g^\pmi_r - \overline c_0$ 
and compactly supported zeroth order terms $B_i$. 
For $i=1$ the first error type is given by 
\begin{align*}  
& E'_1(r,\xi_1,\xi_2)   \\
&:= 
 B^{[R-1,\infty)}  \bigl(J(\overline c_0)-J(\g^{\mi}_r) \bigr) \, \partial_t  \xi_1  
\;+\;
B^{[R-1,\infty)}  \bigl( F_{\overline c_0} - F_{\g^{\mi}_r} \bigr) \,\xi_1 \\
&
\;+\; B^{[R-1,R+1]}  \bigl( F_{\g^{\mi}_r} - F_{\overline c_0} \bigr)\,  \tau_{\mi 2R}\,\xi_2 
\;+\;  B^{[R-1,R+1]} \bigl(J(\g^{\mi}_r)-J(\overline c_0) \bigr)\, \tau_{\mi 2R}\,\partial_t \xi_2 . 
\end{align*}
Note that this is supported in $[R-1,\infty)\times S^1$ and on $[R+1,\infty)\times S^1$ simplifies due to $B^{[R-1,\infty)}|_{[R+1,\infty)\times S^1} \equiv 1$ to  
\begin{align}\label{eq:E1simple}  
E'_1(r,\xi_1,\xi_2)|_{[R+1,\infty)\times S^1}   
&= 
 \bigl(J(\overline c_0)-J(\g^{\mi}_r) \bigr) \, \partial_t  \xi_1  
\;+\;
  \bigl( F_{\overline c_0} - F_{\g^{\mi}_r} \bigr) \,\xi_1   .
\end{align}
Analogously, $E'_2(r,\xi_1,\xi_2)$ is supported in $(- \infty,\mi R+1]\times S^1$ and simplifies to 
\begin{align*}
E'_2(r,\xi_1,\xi_2)|_{(\mi \infty,\mi R-1]\times S^1}   
&= 
 \bigl(J(\overline c_0)-J(\g^\pl_r) \bigr) \, \partial_t  \xi_2  
\;+\;
  \bigl( F_{\overline c_0} - F_{\g^\pl_r} \bigr) \,\xi_2  . 
\end{align*}
The second error types for $i=1,2$ can be rewritten as
\begin{align}
& B_1(r,\xi_1,\xi_2) := B_s^{[R-1,R+1]} \, \tau_{\mi R} \bigl((2\b - 1) \tau_R \xi_1 -  \tau_{\mi R}\, \xi_2 \bigr)   \nonumber \\
&\qquad\qquad\; =
\tau_{\mi R}\Bigl( \tfrac{\partial_s \b (2\b^2 -1)}{(\b^2 + (1-\b)^2)^2} \oplus_R\!(\xi_1,\xi_2) 
+ \tfrac{\partial_s \b (2\b^2 - 4\b +1)}{(\b^2 + (1-\b)^2)^2}\ominus_R\!(\xi_1,\xi_2) \Bigr) ,
\label{eq:Bs1} \\
& B_2(r,\xi_1,\xi_2) := B_s^{[\mi R-1,\mi R+1]} \, \tau_{R} \bigl(  (2\b - 1) \tau_{\mi R}\, \xi_2 + \tau_R \xi_1 \bigr)  \nonumber \\
&\qquad\qquad\; =
\tau_{R}\Bigl(\tfrac{\mi \partial_s \b (2\b^2 - 4\b +1)}{(\b^2 + (1-\b)^2)^2}
 \oplus_R\!(\xi_1,\xi_2) + \tfrac{\partial_s \b (2\b^2 -1) }{(\b^2 + (1-\b)^2)^2} \ominus_R\!(\xi_1,\xi_2) \Bigr). \nonumber
\end{align}

\subsection{The filled Floer section is sc-Fredholm} \label{ss:proof}

To finish the proof of the Fredholm property in Theorem~\ref{mainthm} this section verifies the requirements (ii-a), (ii-b), and (iii) of Definition~\ref{def:scfred} for the map $\Phi:[0,\eps)\times\B\to \E$  given by \eqref{Phi} and its partial differential $\rD_\B\Phi$ that we calculated in \eqref{DPhi}.
Throughout, we will denote by $C_m>0$ all constants that depend only on $m\in\N_0$, even if they change from line to line. We will also use the convention that norms $\|\ldots\|$ without a subscript are the same as the norm specified at the end of the line.

\begin{lemma} \label{le1}
Property (ii-a) of Definition~\ref{def:scfred} holds. That is, given $m\in\N$, $w\in[0,\d)$, and $\theta>0$ there exists $\eps>0$ so that for all $r\in [0, \eps)$ and $\ul e, \ul e'\in B_{m+1}$ with $\|e_i\|_{H^{m+1,w}}, \|e_i'\|_{H^{m+1,w}} < \eps$, and $\uxi=(\xi_1,\xi_2) \in B_m$ we have 
$$
\bigl\| \rD_{\B}\Phi(r,\ul e) \uxi  - \rD_{\B}\Phi(r,\ul e') \uxi \bigr\|_{H^{m,w}}  \le \theta \bigl( \| \xi_1 \|_{H^{m+1,w}} + \| \xi_2 \|_{H^{m+1,w}} \bigr)  .
$$
\end{lemma}
\begin{proof}
To begin, note that the left hand side splits into similar terms for $i=1,2$,
\begin{align*}
 \bigl\| \rD_{\B}\Phi(r,\ul e) \uxi  - \rD_{\B}\Phi(r,\ul e') \uxi \bigr\|_{H^{m,w}} 
=
\textstyle\sum_{i=1,2} \bigl\| E_i(r,\ul e,\uxi)  - E_i(r,\ul e',\uxi) \bigr\|_{H^{m,w}} 
\end{align*}
We establish the estimate representatively for the $i=1$ term, starting from
\begin{align*}
&\bigl\| E_1(r,\ul e,\uxi)  - E_1(r,\ul e',\uxi) \bigr\|_{H^{m,w}} \\
&
\le\; 
C_m \Bigl(
 \bigl\| \bigl( J(\g^{e\,\mi}_r)-J(\g^{e'\,\mi}_r) \bigr) \partial_t  \xi_1 \bigr\|
+ \bigl\| \bigl( F_{\g^{e\, \mi}_r} - F_{\g^{e'\,\mi}_r} \bigr) \xi_1 \bigr\|_{H^{m,w}((\mi\infty,R+1]\times S^1)}   \\
&\quad\qquad\quad\qquad\qquad\quad\;\;
+ 
\bigl\| \bigl( J(\g^{e\,\mi}_r)-J(\g^{e'\,\mi}_r) \bigr) \tau_{\mi 2R}\,\partial_t \xi_2 \bigr\|_{H^{m,w}([R-1,R+1]\times S^1)} \\
&\quad\qquad\qquad\qquad\qquad\qquad\quad\;
+  \bigl\| \bigl( F_{\g^{e\,\mi}_r} - F_{\g^{e'\,\mi}_r} \bigr) \tau_{\mi 2R}\,\xi_2 \bigr\|_{H^{m,w}([R-1,R+1]\times S^1)}  \Bigr) 
\end{align*}
Here the $\cC^m$-norms of the cutoff functions $\bigl\| B^{(\mi \infty,R+1]} \bigr\|_{\cC^m} = \bigl\| \tfrac{\b^2}{\b^2 + (1-\b)^2} \bigr\|_{\cC^m}$ and $\bigl\| B^{[R-1,R+1]} \bigr\|_{\cC^m} = \bigl\| \tfrac{\b(1-\b)}{\b^2 + (1-\b)^2} \bigr\|_{\cC^m}$ are bounded independently of $R$ since, up to a shift, they are given in terms of the fixed function $\b\in\cC^\infty(\R,[0,1])$. The second of these functions is supported in $[R-1,R+1]$ since $\b(1-\b)$ is supported in $[\mi 1,1]$. Now we use the Sobolev embeddings $H^{m+1}\hookrightarrow \cC^{m-1}$ and $H^1\hookrightarrow L^4$ on compact $2$-dimensional domains to estimate
\footnote{
This is the one estimate in the proof of Theorem~\ref{mainthm} where we cannot work with the polyfold version of a Cauchy-Riemann operator $H^1\to H^0$, corresponding to $m=0$ in this Lemma. This is because $H^k$ is closed under multiplication (on compact domains) only for $k>1$. 
More precisely, the case $m=0$ would require an estimate of the error term 
$\bigl\| (\rD_{\g} J)(\xi) \partial_t \g - (\rD_{\g'} J)(\xi) \partial_t \g' \bigr\|_{L^2([\mi 1,1]\times S^1)}$
in terms of $\| \g-\g' \|_{H^1}, \|\xi\|_{H^1}$. 

In the Morse setting, this term is absent and we have Sobolev embeddings $H^m\hookrightarrow\cC^m$, so that this Lemma also holds for $m=0$, although a polyfold setup works with $m\ge 1$ due to the slicing condition in Remarks~\ref{rmk:gen2}, \ref{rmk:gen4}.
}

\begin{align*}
& \tfrac 1{C_m} \bigl\| E_1(r,\ul e,\uxi)  - E_1(r,\ul e',\uxi) \bigr\|_{H^{m,w}} \\
&\quad\;\le\; 
\bigl\| \g^{e\,\mi}_r \! - \g^{e'\,\mi}_r \bigr\|_{\cC^{m-1}} \|  \xi_1\| _{H^{m+1,w}}   
+ \bigl\| \nabla^m( \g^{e\,\mi}_r \!- \g^{e'\,\mi}_r ) \bigr\|_{L^4} \| e^{w \eta} \xi_1\| _{W^{1,4}}  \nonumber\\
&\quad\quad + \bigl\|  \partial_t\g^{e\,\mi}_r \!- \partial_t\g^{e'\,\mi}_r  \bigr\|_{H^{m}} \| e^{w \eta} \xi_1\| _{\cC^0 }  \nonumber\\
&\quad\quad + \bigl\| \g^{e\,\mi}_r \! - \g^{e'\,\mi}_r \bigr\|_{\cC^{m-1}([R-1,R+1]\times S^1)} \| \tau_{\mi 2R}\,\xi_2 \| _{H^{m+1,w}([R-1,R+1]\times S^1)} \\
&\quad\quad + \bigl\| \nabla^m( \g^{e\,\mi}_r \!- \g^{e'\,\mi}_r ) \bigr\|_{L^4([R-1,R+1]\times S^1)} \| e^{w \eta} \tau_{\mi 2R}\,\xi_2\| _{W^{1,4}([R-1,R+1]\times S^1)} \\
&\quad\quad + \bigl\| \partial_t\g^{e\,\mi}_r \!- \partial_t\g^{e'\,\mi}_r ) \bigr\|_{H^{m}([R-1,R+1]\times S^1)} \| e^{w \eta} \tau_{\mi 2R}\,\xi_2\| _{\cC^0([R-1,R+1]\times S^1)} \\
&\quad\;\le\; C_m \bigl\| \g^{e\,\mi}_r \! - \g^{e'\,\mi}_r \bigr\|_{H^{m+1}} \bigl(  \|  \xi_1\| _{H^{m+1,w}}  + \| \tau_{\mi 2R}\xi_2 \| _{H^{m+1,w}([R-1,R+1]\times S^1)} \bigr).
\end{align*}
Finally, we need to estimate the effect of the weight functions $e^{w \eta}$ for the fixed smooth function $\eta$ with $\eta(s)=|s|$ for $|s|\ge 1$. Note here that all derivatives of $\eta$ are bounded so that we can bound $\| e^{\mi  w\eta}\|_{\cC^{m+1}(I)}\leq \eta_m \sup_{s\in I} e^{\mi w |s|}\leq \eta_m$ on any interval $I\subset\R$ in terms of some constants $\eta_m\ge 1$. 
With that we can estimate
\begin{align*}
& \bigl\| \g^{e\,\mi}_r - \g^{e'\,\mi}_r  \bigr\|_{H^{m+1}}  
\;=\; 
\| \tau_{\mi R}\b \cdot (e_1-e'_1) + (1- \tau_{\mi R}\b) \tau_{\mi 2R}(e_2-e'_2) \|_{H^{m+1}} \\
&\qquad\qquad\qquad\qquad\;\le\; 
\|\beta \|_{\cC^{m+1}} \| e^{\mi  w\eta}\|_{\cC^{m+1}} \| e^{w \eta} ( e_1-e'_1 ) \|_{H^{m+1}}  \\
&\qquad\qquad\qquad\qquad\qquad +  
\|1-\beta \|_{\cC^{m+1}}  \| e^{\mi  w\eta}\|_{\cC^{m+1}} \| e^{w \eta} (e_2-e'_2 ) \|_{H^{m+1}} \\
&\qquad\qquad\qquad\qquad\;\le\; 
C_m \bigl(
\| e_1 - e_1' \|_{H^{m+1,w}} +  \| e_2 - e_2' \|_{H^{m+1,w}} \bigr), \\
&\| \tau_{\mi 2R}\xi_2 \| _{H^{m+1,w}([R-1,R+1]\times S^1)} \\
&\qquad\quad=\; 
\| e^{w \tau_{2R}\eta }  \xi_2 \| _{H^{m+1}([\mi R-1,\mi R+1]\times S^1)}  \\
&\qquad\quad\le\; 
\| e^{w (\tau_{2R}\eta - \eta) } \|_{\cC^{m+1}([\mi R-1,\mi R+1])}  \|e^{w\eta}  \xi_2 \|_{H^{m+1}}  
\;\le\; 
C_m \| \xi_2 \| _{H^{m+1,w}}  
\end{align*}
Here the last constant $C_m$ arises from the maximal derivative up to order $m$ of the function 
$s\mapsto e^{w (\eta(R+s) - \eta(\mi R+s))}$ for $s\in[\mi 1,1]$, which for $R\ge 2$ simplifies to 
$s\mapsto e^{2 w s}$. 
Plugging this into the above estimate for $E_1$
and its analogue for $E_2$ we obtain 
\begin{align*}
& \bigl\| \rD_{\B}\Phi(r,\ul e) \uxi  - \rD_{\B}\Phi(r,\ul e') \uxi \bigr\|_{H^{m,w}} \\
&\;\le\; 
C_m \bigl( \| e_1 - e'_1 \|_{H^{m+1,w}} + \| e_2 - e'_2 \|_{H^{m+1,w}} \bigr)
\bigl(  \|\xi_1\| _{H^{m+1,w}} +  \|\xi_2\| _{H^{m+1,w}}  \bigr)  .
\end{align*}
This proves the lemma with $\eps=\theta / 4 C_m$. 
\end{proof}

Before proving the Fredholm properties (ii-b) and (iii) we note a key estimate.

\begin{lemma} \label{lem:est}
For any $m\in\N_0$ and weight $w\in\R$ there is a constant $C>0$ 
so that for all $\uxi=(\xi_1,\xi_2)\in B_m$ and  $r>0$ with $R=e^{1/r}\ge 2$ we have 
\begin{align*}
& \bigl\| \rD_{\B}\Phi(r,0) \uxi -  \rD_{\B}\Phi(0,0)\uxi - B_1(r,\uxi) - B_2(r,\uxi) \bigr\|_{H^{m,w}} \\
&\qquad\le
C_m
\bigl( \| \g_1 \|_{\cC^{m+1}([R-1,\infty)\times S^1) } 
+\|  \g_2 \|_{\cC^{m+1}((\mi \infty,\mi R+1]\times S^1) } \bigr) 
\| \uxi \|_{H^{m+1,w}} 
\end{align*}
\end{lemma}
\begin{proof}
To begin, note that all summands in the operator difference 
\begin{align*}
& \bigl\| \rD_{\B}\Phi(r,0) \uxi -  \rD_{\B}\Phi(0,0)\uxi - B_1(r,\uxi)-B_2(r,\uxi) \bigr\|_{H^{m,w}} \\
&\qquad =\phantom{+}
\bigl\| \bigl( J(\g^\mi_{r}) - J(\g_1)\bigr) \partial_t \xi_1 - \bigl( F_{\g^\mi_{r}} - F_{\g_1}\bigr) \xi_1 + E'_1(r,\uxi)  \bigr\|_{H^{m,w}} \\
&\qquad\phantom{=}\, + 
\bigl\| \bigl( J(\g^\pl_{r}) - J(\g_2)\bigr) \partial_t \xi_2 - \bigl( F_{\g^+_{r}} - F_{\g_2}\bigr) \xi_2 + E'_2(r,\uxi)  \bigr\|_{H^{m,w}}.
\end{align*}
are supported outside of ${[- R+1,R-1]\times S^1}$ since  $(\g^\mi_{r}- \g_1)_{(\mi \infty,R-1]\times S^1}=0$ and $(\g^\pl_{r} - \g_2)|_{[\mi R+1,\infty)\times S^1}=0$.
Moreover, outside of ${[- R-1,R+1]\times S^1}$ the error terms $E'_i:=E'_i(r,\uxi)$ simplify as in \eqref{eq:E1simple} to yield
\begin{align*}
& \bigl\| \rD_{\B}\Phi(r,0) \uxi -  \rD_{\B}\Phi(0,0) \uxi  - B_1(r,\uxi)-B_2(r,\uxi) \bigr\|_{H^{m,w}((\R\less [\mi R-1,R+1])\times S^1)} \\
&=  \phantom{+} \bigl\| \bigl( J(\g^{\mi}_{r}) - J(\g_1)\bigr) \partial_t \xi_1 - \bigl( F_{\g^\mi_{r}} - F_{\g_1}\bigr) \xi_1 + E'_1  \bigr\|_{H^{m,w}([R+1,\infty)\times S^1)} \\
&\quad +
\bigl\| \bigl( J(\g^{+}_{r}) - J(\g_2)\bigr) \partial_t \xi_2 - \bigl( F_{\g^\pl_{r}} - F_{\g_2}\bigr) \xi_2 + E'_2  \bigr\|_{H^{m,w}((\mi \infty,\mi R-1]\times S^1)} \\
&= \phantom{+}
  \bigl\| \bigl(J(\overline c_0) - J(\g_1)\bigr) \partial_t \xi_1 - \bigl( F_{\overline c_0} - F_{\g_1}\bigr) \xi_1  \bigr\|_{H^{m,w}([R+1,\infty)\times S^1)}\\
 &\quad +
 \bigl\| \bigl(J(\overline c_0) - J(\g_2)\bigr) \partial_t \xi_2 - \bigl( F_{\overline c_0} - F_{\g_2}\bigr) \xi_2  \bigr\|_{H^{m,w}((\mi \infty,\mi R-1]\times S^1)}\\
&\le
 C_m \bigl( \| \overline c_0 - \g_1\bigr\|_{\cC^{m+1}([R+1,\infty)\times S^1) }  +
  \| \overline c_0 - \g_2\bigr\|_{\cC^{m+1}((\mi \infty,\mi R-1]\times S^1) } \bigr) \| \uxi \|_{H^{m+1,w}}  
\end{align*}
On the remaining finite cylinder $[-R-1,R+1]\times S^1$ we have the estimate 
\begin{align*}
& \bigl\| \rD_{\B}\Phi(r,0) \uxi -  \rD_{\B}\Phi(0,0)\uxi  - B_1(r,\uxi)-B_2(r,\uxi) \bigr\|_{H^{m,w}([\mi R-1,R+1]\times S^1)} \\
&\quad= \bigl\| \bigl( J(\g^\mi_{r}) - J(\g_1)\bigr) \partial_t \xi_1 
- \bigl( F_{\g^\mi_{r}} - F_{\g_1}\bigr) \xi_1 + E'_1 \, \bigr\|_{H^{m,w}([R-1,R+1]\times S^1)} \\
&\qquad + \bigl\| \bigl( J(\g^\pl_{r}) - J(\g_2)\bigr) \partial_t \xi_2 
- \bigl( F_{\g^\pl_{r}} - F_{\g_2}\bigr) \xi_2 + E'_2 \, \bigr\|_{H^{m,w}([\mi R-1,\mi R+1]\times S^1)} \\
&\quad\le
C_m
\bigl( \| \g^\mi_{r} - \g_1 \| +  \| \overline c_0 - \g^\mi_{r} \|_{\cC^{m+1}([R-1,R+1]\times S^1) } \bigr) \\
&\quad\qquad\qquad\cdot
\bigl(
\| \xi_1 \| + \| \tau_{\mi 2R}\,\xi_2  \|_{H^{m+1,w}([R-1,R+1]\times S^1)} \bigr) \\
&\qquad+ C_m
\bigl( \| \g^\pl_{r} - \g_2 \| +  \| \overline c_0 - \g^\pl_{r} \|_{\cC^{m+1}([\mi R-1,\mi R+1]\times S^1) } \bigr) \\
&\qquad\qquad\qquad\cdot
\bigl(
\| \xi_2 \| + \| \tau_{2R}\,\xi_1  \|_{H^{m+1,w}([\mi R-1,\mi R+1]\times S^1)} \bigr) .
\end{align*}
Now recall $\overline c_0 \equiv 0$, $\g^-_0=\g_1$, and $\g^+_0=\g_2$, so we can estimate
\begin{align*}
&\| \g^\mip_{r} - \g^\mip_0 \| + \| \g^\mip_{r} - \overline c_0 \|_{\cC^{m+1}([\pmi R-1,\pmi R+1]\times S^1) } \\
&\le
2\| \b \tau_{R}\g_1 \| + 2 \| (\b-1) \tau_{\mi  R} \g_2 \| + \| (1-\b) \tau_{R}\g_1\| + \| \b  \tau_{\mi R}\g_2\|_{\cC^{m+1}([\mi 1,1]\times S^1) }   \\
& \le
C_m \bigl( \| \g_1 \|_{\cC^{m+1}([R-1,R+1]\times S^1) } 
+\|  \g_2 \|_{\cC^{m+1}([\mi R-1,\mi R+1]\times S^1) } \bigr) .
\end{align*}
Finally, for $R \ge 2$ we bound the effect of the shift in $\tau_{\mip 2R}\,\xi |_{[\pmi R-1,\pmi R+1]\times S^1}$ by 
\begin{align*}
& \| \tau_{\mip 2R}\,\xi  \|_{H^{m+1,w}([\pmi R-1,\pmi R+1]\times S^1)}
\;=\;
 \|  e^{w |s\pmi 2R|} \xi  \|_{H^{m+1}([\mip R-1,\mip R+1]\times S^1)}\\
& \;=\;
 \|  e^{2w (R\pmi s )} e^{\mip w s }  \xi  \|_{H^{m+1}([\mip R-1,\mip R+1]\times S^1)} 
\;\leq\; \| e^{\pmi 2w s} \|_{\cC^{m+1}([\mi 1,1])} \|\xi  \|_{H^{m+1,w}}  
\end{align*}
Putting all these together yields what remained to be established,
\begin{align*}
& \bigl\| \rD_{\B}\Phi(r,0) \uxi -  \rD_{\B}\Phi(0,0) \uxi   - B_1(r,\uxi)-B_2(r,\uxi) \bigr\|_{H^{m,w}([\mi R-1,R+1]\times S^1)} \\
& \leq 
C_{m,w} \bigl( \| \g_1 \|_{\cC^{m+1}([R-1,R+1]\times S^1) } 
+\|  \g_2 \|_{\cC^{m+1}([\mi R-1,\mi R+1]\times S^1) } \bigr)
\| \uxi \|_{H^{m+1,w}}. 
\end{align*}
\end{proof}

\begin{lemma} \label{le2}
Property (ii-b) of Definition~\ref{def:scfred} holds without subsequences both for the sc-Banach spaces $\B,\E$ in \eqref{eq:BE} and their extensions \eqref{eq:BEext}. 
That is, for fixed $m\in\N_0$, $w\in (-\d,\d)$, and sequences $0<r^\nu\to 0$, $\uxi^\nu\in B_m$ with $\|\uxi^\nu\|_{H^{m+1,w}}< 1$ and ${\| \rD_{\B}\Phi(r^\nu,0) \uxi^\nu \|_{H^{m,w}} \to 0}$, we also have $\| \rD_{\B}\Phi(0,0) \uxi^\nu \|_{H^{m,w}}\to 0$.
\end{lemma}
\begin{proof}
Let $r^\nu,\uxi^\nu=(\xi_1^\nu,\xi_2^\nu)$ be such sequences and set $R^\nu:=e^{1/r^\nu}\to \infty$, then it suffices to show
$\bigl\| \rD_{\B}\Phi(r^\nu,0) \uxi^\nu -  \rD_{\B}\Phi(0,0)\uxi^\nu \bigr\|_{H^{m,w}}\underset{\nu\to\infty}\to 0$. 
Lemma~\ref{lem:est} gives\footnote{The estimate for general limit orbits $c_0\not\equiv 0$ involves terms $\| \g_1 - \overline c_0\|$ and $\| \g_2 - \overline c_0 \|$.}
\begin{align*}
& \bigl\| \rD_{\B}\Phi(r^\nu,0) \uxi^\nu -  \rD_{\B}\Phi(0,0)\uxi^\nu - B_1(r^\nu,\uxi^\nu) - B_2(r^\nu,\uxi^\nu) \bigr\|_{H^{m,w}} \\
&\qquad\le
C_m \bigl( \| \g_1 \|_{\cC^{m+1}([R^\nu-1,\infty) \times S^1) } +  \| \g_2 \|_{\cC^{m+1}((\mi \infty,\mi R^\nu+1] \times S^1) } \bigr)  \| \uxi^\nu \|_{H^{m+1,w}} 
\end{align*}
These norms converges to $0$ due to the bounds $\| \xi^\nu \|_{H^{m+1,w}}<1$ and the limits $\lim_{s\to\pm\infty}\g_i(s,t)= c_0$ which are uniform with all derivatives.
It remains to show convergence to $0$ of $B_i(r^\nu,\uxi^\nu)$ which is supported on $[\pm R^\nu-1,\pm R^\nu+1]\times S^1$.
Using \eqref{eq:Bs1} we can estimate these by\footnote{
The error terms $B^i(r^\nu,\xi_1^\nu,\xi^2_\nu)$ arising from the derivative of the cutoff function $\partial_s\beta$ could alternatively be bounded in terms of $\| \xi_i^\nu \|_{H^{m,w}([\pmi R^\nu-1,\pmi R^\nu+1]\times S^1)}$ and 
$\| \tau_{\mip 2R^\nu}\,\xi_i  \|_{H^{m,w}([\pmi R^\nu-1,\pmi R^\nu+1]\times S^1)} $. However -- despite compact Sobolev embeddings $H^{m+1}\hookrightarrow H^m$ on finite cylinders -- these norms generally do not converge to $0$. Consider e.g.\ $\xi_1^\nu (s,t) = e^{\mi w s} \psi(s-R^\nu)$ with a smooth bump function $\psi$ supported in $[\mi 1,1]$.}

\begin{align}
& \bigl\| B_1(r^\nu,\uxi^\nu) + B_2(r^\nu,\uxi^\nu) \bigr\|_{H^{m,w}} \nonumber \\
&\qquad\leq  C_m \bigl(  
\|\tau_{\mi R^\nu}  \oplus_{R^\nu}\!\!(\uxi^\nu) \|
+ \|\tau_{\mi R^\nu}  \ominus_{R^\nu}\!\!(\uxi^\nu) \|_{H^{m,w}([R^\nu-1,R^\nu+1]\times S^1)} \label{eq:Best}\\
&\qquad\qquad\quad  +  
\|\tau_{R^\nu}  \oplus_{R^\nu}\!\!(\uxi^\nu) \|
+ \|\tau_{R^\nu}  \ominus_{R^\nu}\!\!(\uxi^\nu) \|_{H^{m,w}([\mi R^\nu-1,\mi R^\nu+1]\times S^1)} \bigr) \nonumber \\
&\qquad\leq  C_m e^{w R^\nu}   \bigl(  
\| \oplus_{R^\nu}\!\!(\uxi^\nu) \|_{H^m([\mi 1,1]\times S^1)}  
+ \|\ominus_{R^\nu}\!\!(\uxi^\nu) \|_{H^m([\mi 1,1]\times S^1)}  \bigr), \nonumber
\end{align}
where the last step estimates the effect of shifts and weight functions by 
$$
\|\tau_{\mip R} \zeta \|_{\scriptscriptstyle H^{m,w}([\pmi R-1,\pmi R+1]\times S^1)}
= \| e^{w \tau_{\pm R}\eta}   \zeta \|_{\scriptscriptstyle H^m([\mi 1,1]\times S^1)}
 \le C_m e^{w R} \| \zeta \|_{\scriptscriptstyle H^m([\mi 1,1]\times S^1)}  
$$
We will now bound $\oplus_{R^\nu}(\uxi^\nu), \ominus_{R^\nu}(\uxi^\nu)$ in terms of $\rD_{\B}\Phi(r^\nu,0) (\uxi^\nu)\to 0$ by using the invertibility of the filler $\rD_{\overline c_0} \pbar$ which we establish in Lemma~\ref{lem:invert} below. Recall from \eqref{defDPhi} that linearizing the defining equation \eqref{Phi} for $\Phi$
yields an identity for $\rD_{\B}\Phi(r,0) (\uxi)=:(\phi_1,\phi_2)$
\begin{align*}
\left( \begin{aligned}
(\rD_{\g_r} \pbar) \oplus_R \! (\uxi)  
\\
(\rD_{\overline c_0} \pbar) \ominus_R \! (\uxi)  
\end{aligned}
\right)
&=
(\widehat\oplus_R \times \widehat\ominus_R) \rD_\B\Phi(r,0)(\uxi) 
=
\left( \begin{aligned}
\beta \tau_R \phi_1 + (1-\beta) \tau_{\mi R} \phi_2
\\
(\beta-1) \tau_R \phi_1 + \beta \tau_{\mi R} \phi_2
\end{aligned}
\right) 
\end{align*}
Using the invertibility of $\rD_{\overline c_0} \pbar:H^{m+1,w}(\R\times S^1)\to H^{m,w}(\R\times S^1)$ this yields 
\begin{align*}
& C_m^{\mi 1} \| \ominus_R \! (\uxi) \|_{H^m([\mi 1,1]\times S^1)}
\leq C_m^{\mi 1}  \|  \ominus_R \! (\uxi) \|_{H^{m+1,w}(\R\times S^1)} \\
&\qquad \leq 
\bigl\| (\rD_{\overline c_0} \pbar) \ominus_R \! (\uxi)  \bigr\|_{H^{m,w}(\R\times S^1)} 
= \| (\beta-1) \tau_R \phi_1 + \beta \tau_{\mi R} \phi_2  \|_{H^{m,w}(\R\times S^1)} \\
&\qquad\leq 
 \| \b \|_{\cC^m} \bigl( \| \tau_R \phi_1 \|_{H^{m,w}([\mi 1,\infty)\times S^1)} + \| \tau_{\mi R} \phi_2 \|_{H^{m,w}((\mi \infty,1]\times S^1)} \bigr)\\
&\qquad\leq 
 \| \b \|_{\cC^m} \bigl( \|e^{w(\tau_{\mi R}\eta - \eta)}\|_{\cC^m([R-1,\infty)\times S^1))}
 \| e^{w\eta} \phi_1 \|_{H^{m}([R-1,\infty)\times S^1)}  \\
&\qquad\phantom{\leq    \| \b \|_{\cC^m} \bigl( }
+  \|e^{w(\tau_R\eta -\eta)}\|_{\cC^m(\mi \infty,\mi R+1]\times S^1))}
\| e^{w\eta} \phi_2 \|_{H^{m}((\mi \infty,\mi R+1]\times S^1)} \bigr)\\
&\qquad\leq 
C_m e^{\mi w R} \bigl (  \|\phi_1\|_{H^{m,w}} + \|\phi_2\|_{H^{m,w}}  \bigr)
= C_m e^{\mi w R} \|\rD_{\B}\Phi(r,0) \uxi\|_{H^{m,w}}.
\end{align*}
In our estimate \eqref{eq:Best} for $\|B_i(r^\nu,\uxi^\nu)\|_{H^{m,w}}$, the above
bounds the second term $e^{w R^\nu} \| \ominus_{R^\nu} \! (\xi^\nu_1,\xi_2^\nu) \|_{H^{m}([\mi 1,1]\times S^1)}\leq C_m \|\rD_{\B}\Phi(r^\nu,0) \uxi^\nu \|_{H^{m,w}}$ by a quantity that converges to $0$ by assumption. 
To prove convergence of the first term in \eqref{eq:Best},
$e^{w R^\nu} \| \oplus_{R^\nu}\!(\uxi^\nu) \|_{H^m([\mi 1,1]\times S^1)} \to 0$, 
we need some preparations:

We fix a family of cutoff functions $\psi_T\in\cC^\infty(\R,[0,1])$ for $T\ge 2$ supported in $[\mi T,T]$ with $\psi|_{[\mi T+1, T-1]}\equiv 1$ with uniformly bounded derivatives.
Next, we use a weight $w'\in(-\d,-|w|)$ in Lemma~\ref{lem:invert} which gives rise to a constant $\eps'>0$. Then for sufficiently large $R=e^{1/r}$, and $T\in [2, \frac 12 R]$ we wish to pick an extension $\g_{r,T}\in\cC^\infty(\R\times S^1,\C^n)$ of $\g_{r,T}|_{[\mi T,T]}=\g_r|_{[\mi T,T]}$ with ${\| \g_{r,T} - \overline c_0 \|_{\cC^{m+1}}<\eps'}$.
Such extensions require $\| \g_r - \overline c_0 \|_{\cC^{m+1}([\mi T,T]\times S^1)}$ to be sufficiently small -- depending on $\eps'$ and $m$, but independent of the length $2T$ of the given interval.
Large $R\ge 2T$ will guarantee this since uniform convergence $\g_i \to \overline c_0 =0$ for $s\to\pm \infty$ yields
\begin{align*}
& \| \g_r - \overline c_0 \|_{\cC^{m+1}([\mi T,T]\times S^1)} 
= \| \b \tau_R \g_1 + (1-\b)\tau_{\mi R}\g_2 \|\|_{\cC^{m+1}([\mi T,T]\times S^1)} \\
&\qquad \leq \|\b\|_{\cC^{m+1}}\bigl( \|\g_1 \|_{\cC^{m+1}([R-T,R+T]\times S^1)} +  \|\g_2 \|_{\cC^{m+1}([\mi R-T,\mi R+T]\times S^1)} \bigr) \\
&\qquad\leq \|\b\|_{\cC^{m+1}}\bigl( \|\g_1 \|_{\cC^{m+1}([R/2,\infty)\times S^1)} +  \|\g_2 \|_{\cC^{m+1}((\mi \infty,\mi R/2]\times S^1)} \bigr) \;\underset{R\to\infty}{\longrightarrow}\; 0.
\end{align*}
After these preparations we can use the estimate for the linearized Floer operator $\rD_{\g_{r,T}} \pbar: H^{m+1,w'}(\R\times S^1)\to H^{m,w'}(\R\times S^1)$ from Lemma~\ref{lem:invert} to bound
\begin{align*}
& C_m^{\mi 1} \| \psi_T \oplus_R \! (\uxi) \|_{H^{m+1,w'}} \\
&\leq 
\bigl\| (\rD_{\g_{r,T}} \pbar) \bigl( \psi_T \oplus_R \! (\uxi) \bigr) \bigr\|_{H^{m,w'}} 
= \bigl\| (\rD_{\g_r} \pbar) \bigl( \psi_T \oplus_R \! (\uxi) \bigr) \bigr\|_{H^{m,w'}([\mi T,T]\times S^1)} \\
&\leq 
\bigl\| \tfrac{\rd}{\rd s}\psi_T \cdot \oplus_R  (\uxi) \bigr\|_{H^{m,w'}}
+ \| \psi_T\bigl( \beta \tau_R \phi_1 + (1-\beta) \tau_{\mi R} \phi_2 \bigr) \|_{H^{m,w'}([\mi T,T]\times S^1)} \\
&\leq 
 \| \psi_T \|_{\cC^{m+1}} \| \b \|_{\cC^m}
\bigl( \| \tau_R \xi_1 \|_{H^{m,w'}(\supp\psi'\times S^1)} +  \| \tau_{\mi R} \xi_2  \|_{H^{m,w'}(\supp\psi'\times S^1)} \\
& \phantom {\leq\| \psi_T \|_{\cC^{m+1}} \| \b \|_{\cC^m} \bigl(  } 
+  \| \tau_R \phi_1 \|_{H^{m,w'}([\mi T,T]\times S^1)} + \| \tau_{\mi R} \phi_2 \|_{H^{m,w'}([\mi T,T]\times S^1)} \bigr)\\
&\leq 
C_m e^{\mi w R} \bigl (  e^{(w'+|w|)T } \tsum_{i=1,2} \|\xi_i\|_{H^{m,w}} +
e^{|w| T} \tsum_{i=1,2} \|\phi_i\|_{H^{m,w}} \bigr) .
\end{align*}
Here the last step uses the general comparison between weighted Sobolev norms on finite cylinders for $R\ge 2T \ge T+2$
\begin{align*}
 \| \tau_{\pmi R} \xi \|_{H^{m,w'}(\supp\psi'\times S^1)} 
& \le \| e^{w'\eta}\|_{\cC^m(\supp\psi')} \| \xi \|_{H^{m}([\pmi R-T,\pmi R+T]\times S^1)}\\
& \le \eta_m e^{w' (T-1)} \| e^{\mi w\eta}\|_{\cC^m([\pmi R-T,\pmi R+T]} \| e^{w \eta} \xi \|_{H^{m}} \\
&\le \eta_m^2 e^{w' (T-1)} e^{\mi w R + |w| T} \| \xi \|_{H^{m,w}} ,\\
 \| \tau_{\pmi R} \phi \|_{H^{m,w'}([\mi T,T]\times S^1)} 
& \le \| e^{w'\eta}\|_{\cC^m([\mi T,T]} \| \phi \|_{H^{m}([\pmi R-T,\pmi R+T]\times S^1)}\\
& \le \eta_m \| e^{\mi w\eta}\|_{\cC^m([\pmi R-T,\pmi R+T]} \| e^{w \eta} \phi \|_{H^{m}}\\
&\le \eta_m^2 e^{\mi w R +|w|T} \| \phi \|_{H^{m,w}}  
\end{align*}
Now inserting $r^\nu\to 0$, $\|\xi^\nu_i\|_{H^{m+1,w}}<1$ and appropriate $T^\nu\in[2,R^\nu/2]$ yields
\begin{align*}
&e^{w R^\nu}  
\| \oplus_{R^\nu}\!(\uxi^\nu) \|_{H^m([\mi 1,1]\times S^1)} 
\le e^{w R^\nu}  
\| \psi_{T^\nu} \oplus_{R^\nu}\!(\uxi^\nu) \|_{H^{m+1,w'}(\R\times S^1)}  \\
&\qquad\qquad\qquad\qquad\leq 
C_m^2 \bigl ( 2  e^{(w'+|w|)T^\nu } + e^{|w| T^\nu}
 \| \rD_{\B}\Phi(r^\nu,0) \uxi^\nu \|_{H^{m,w}} \bigr) \;\to\; 0 . 
\end{align*}
Here we need to choose $R^\nu/2 \ge T^\nu\to\infty$ for the first term to converge courtesy of $w'<-|w|$, and this is possible without losing convergence of the second term since $\| \rD_{\B}\Phi(r^\nu,0) \uxi^\nu \|_{H^{m,w}}\to 0$. 
This finishes the last step in the proof of convergence
$\bigl\| \rD_{\B}\Phi(r^\nu,0) \uxi^\nu -  \rD_{\B}\Phi(0,0)\uxi^\nu  \bigr\|_{H^{m,w}} \to 0$ as claimed.
\end{proof}

\begin{lemma}\label{lem:invert}
The linearized Floer operator at $\g\in\cC^\infty(\R\times S^1,\C^n)$ with limits \eqref{eq:exp} is an isomorphism
$\rD_\g \pbar :  H^{m+1,w}(\R\times S^1,\C^n) \to H^{m,w}(\R\times S^1,\C^n)$ if both the weight $|w|$ and the $\cC^{m+1}$ distance between $\g$ and  $\overline c_0$ are sufficently small. More precisely, there is $\delta>0$ so that for every $m\in\N_0$ and $w\in(- \delta,\delta)$ there exist constants $C,\eps>0$ so that for $\|\g-\overline c_0\|_{\cC^{m+1}}<\eps$ we have
\begin{equation}\label{eq:linop}
\| \xi \|_{H^{m+1,w}} \leq C \| (\rD_\g \pbar)\xi\|_{H^{m,w}} 
\qquad\forall  \xi\in H^{m+1,w}(\R\times S^1,\C^n)
\end{equation}
\end{lemma}

\begin{proof}
In our simplified setting, the Hamiltonian orbit $\overline c_0\equiv 0$ is a critical point of a function $H:\C^n\to \R$ which yields the vector field $X=J\nabla H$. It is nondegenerate in the sense that the Hessian of $H$ at $0$ is an isomorphism $\text{\it Hess}_0:\C^n\to \C^n$. In particular, the operator $F_{\overline c_0}$ simplifies to $\xi\mapsto \text{\it Hess}_0\xi$, a self-adjoint operator on $L^2(S^1,\C^n)$. Thus $\rD_{\overline c_0} \pbar = \partial_s + A(s)$ is an operator of Atiyah-Patodi-Singer type, given by a family of self-adjoint operators $A(s)=J(0) \, \partial_t + \text{\it Hess}_0$.
This particular family is constant and invertible, so that by e.g.\ \cite[Prop.3.12]{robbinsalamon} it induces an isomorphism $H^1(\R\times S^1,\C^n)\to H^0(\R\times S^1,\C^n)$. This yields an estimate
$$
\| \xi \|_{H^{1}} \leq C_0 \| (\rD_{\overline c_0} \pbar) \xi \|_{H^0} \qquad\forall \xi\in H^{1}(\R\times S^1,\C^n) .
$$
The generalization to weighted Sobolev spaces $H^{1,w}\to H^{0,w}$ can be achieved via the isometries $H^{m,w}\to H^{m},  \xi \mapsto e^{w\eta}\xi$. Composition with these yields the operator $e^{\mi w\eta} (\partial_s + A) e^{w\eta} = \partial_s + A + w\eta' : H^1\to H^0$.
Similarly, we can write the more general linearized operator as sum
$(\rD_{\g} \pbar) = \rD_{\overline c_0} \pbar  + E_\g$ of the known isomorphism with an error term $E_\g = \bigl(J(\g)-J(\overline c_0)\bigr)\partial_t + F_\g-F_{\overline c_0}$.  Now in each case it remains to establish a bound $\|E\xi\|_{H^0} \leq c \|\xi\|_{H^1}$ with $c<\frac 1{C_0}$ to obtain the estimate $\|\xi\|_{H^1}\leq \frac{C_0}{1-C_0 c}\| (\partial_s+A+E)\xi\|_{H^0}$. 

For transferring the $H^1$ estimate to weight $w\in\R$ we have $\| w \eta'\xi\|_{H^0}\leq |w|\cdot \|\eta'\|_{\cC^{0}} \|\xi\|_{H^0}$. So taking $\delta\leq (2C_0\|\eta'\|_{\cC^{0}})^{\mi 1}$ yields for $m=0$
$$
\| \xi \|_{H^{m+1,w}} \leq 2 C_0 \| (\rD_{\overline c_0} \pbar) \xi \|_{H^{m,w}} \qquad\forall  |w|< \d, \xi\in H^{m+1,w}(\R\times S^1,\C^n) .
$$
For general $m\in\N_0$ this estimate is obtained by summing over the $H^{1,w}$ estimates applied to derivatives, $\nabla^\alpha \xi$, noting that derivatives commute with $\rD_{\overline c_0} \pbar$.
This proves the Lemma for $\g=\overline c_0$ using norms $\|\xi\|_{H^{m+1,w}}=\sum_{|\alpha|\leq m} \|\nabla^\alpha\xi\|_{H^{1,w}}$ and $\|\zeta\|_{H^{m,w}}=\sum_{|\alpha|\leq m} \|\nabla^\alpha\zeta\|_{H^{0,w}}$. These norms are equivalent to other more standard norms via constants depending on $m$ and $w$, so the constant $C=C_{m,w}$ generally will not be uniform, but the allowable weights $w$ are independent of $m$. 

Finally, for general $\g\in\cC^\infty(\R\times S^1,\C^n)$ we bound the error term by
$$
\bigl\| \bigl(J(\g)-J(\overline c_0)\bigr)\partial_t\xi + (F_\g-F_{\overline c_0})\xi \bigr\|_{H^{m,w}}
\leq \| \g- \overline c_0 \|_{\cC^{m+1}} \|\xi\|_{H^{m+1,w}}
$$
to see that $\| \g-\overline c_0 \|_{\cC^{m+1}} < C_{m,w}^{\mi 1}$ guarantees the claimed estimate.
\end{proof}

\begin{lemma} \label{le3}
Property (iii) of Definition~\ref{def:scfred} holds both for the sc-Banach spaces $\B,\E$ in \eqref{eq:BE} and the extensions $\B^0,\E^0$ in \eqref{eq:BEext}.  
In fact, the linearized operator 
$$
\rD_{\B}\Phi(0,0) = \bigl( \partial_s  + J(\g_i) \, \partial_t - \rD_{\g_i} (JX) \bigr)_{i=1,2} : \; \B^0 \;\to\; \E^0
$$ 
is sc-Fredholm as in Definition~\ref{fred op}. For any $m\in\N_0$ the linearized operators $\rD_\B\Phi(r,0): H^{m+1,\d_m}\times H^{m+1,\d_m} \to H^{m,\d_m}\times H^{m,\d_m}$ are classically Fredholm for sufficiently small $r>0$ with the same Fredholm index as $\rD_\B\Phi(0,0)$, and
they are regularizing in the sense that for $(\xi_1,\xi_2)\in H^{1,w}\times H^{1,w}$ with $|w|<\d$ we have 
$$
 \rD_{\B}\Phi(r,0) (\xi_1,\xi_2) \in H^{m,\delta_m}\times H^{m,\delta_m}
 \quad \Longrightarrow \quad  (\xi_1,\xi_2)  \in H^{m+1,\d_m}\times H^{m+1,\d_m} .
$$ 
\end{lemma}
\begin{proof}
To check that $\rD_{\B}\Phi(0,0) = \bigl( \rD_{\g_i} \pbar \bigr)_{i=1,2}:\B^0\to\E^0$ is sc-Fredholm we note that the conditions of Definition~\ref{fred op} follow from the following standard elliptic estimates, regularity, and Fredholm properties for the linearized Floer operators $\rD_{\g} \pbar$ 
at any $\g\in\cC^\infty(\R\times S^1,\C^n)$ with limits $\lim_{s\to\pm\infty}\g(s,\cdot)=c_0$ as in \eqref{eq:exp}. 

\begin{enumerate}
\item
For $m\in\N_0$ and $w\in\R$ the operator $\rD_{\g} \pbar : H^{m+1,w} \to H^{m,w}$ is bounded.
\item
For $\xi\in H^{1,w}$ with $|w|<\d$ and $(\rD_{\g} \pbar) \xi \in H^{m,\d_m}$ we have $\xi\in H^{m+1,\d_m}$. 
To confirm this, we first use elliptic regularity on finite cylinders to deduce $\psi \xi \in H^{m+1,\d_m}$ for any compactly supported smooth cutoff function $\psi$. 
%
Next, we pick $\g_0\in\cC^\infty(\R\times S^1,\C^n)$ so that it coincides with $\g$ on the support of $1-\psi$ and satisfies the condition $\|\g_0-\overline c_0\|_{\cC^{m+1}}<\eps$ from Lemma~\ref{lem:invert}. 
This is possible due to the uniform convergence $\lim_{s\to\pm\infty}\g(s,\cdot)=c_0$ while choosing $\psi|_{[-S,S]}\equiv 1$ for sufficiently large $S$. 
Then these choices yield invertible operators 
$\rD_{\g_0} \pbar : H^{m+1,\d_m}\to H^{m,\d_m}$ and $\rD_{\g_0} \pbar : H^{1,w'}\to H^{0,w'}$ for $w'=\min\{w,\d_m\}$. Here $w'$ fits Lemma~\ref{lem:invert} since $|w|, |\d_m|<\d$.
Now we have $\rD_{\g_0}  \pbar (1-\psi)\xi = \rD_\g  \pbar \xi- \rD_{\g_0} \pbar \psi\xi$, which is of regularity $H^{m,\d_m}$ by hypothesis and the regularity of $\psi\xi$. So we can find $\zeta\in H^{m+1,\d_m}$ with $\rD_{\g_0} \pbar \zeta = \rD_{\g_0}  \pbar (1-\psi)\xi$. 
Moreover, we have $(1-\psi)\xi = \zeta$ since $(1-\psi)\xi - \zeta \in H^{1,w'}$ lies in the trivial kernel of $\rD_{\g_0} \pbar$. Finally, this proves $\xi = \psi\xi + (1-\psi)\xi = \zeta + (1-\psi)\xi \in H^{m+1,\d_m}$.
\item
$\rD_{\g} \pbar : H^{1,w} \to H^{0,w}$ is a Fredholm operator for $|w|<\d$.
This can be proven by weighted versions of the Sobolev estimates in e.g.\ \cite{AD2010,Salamon}. Alternatively, the isomorphisms $H^{m,w}\to H^{m},  \xi \mapsto e^{w\eta}\xi$ transform the operator to $H^1\to H^0$, $\xi \mapsto e^{\mi w\eta} (\rD_{\g} \pbar) (e^{w\eta}\xi) = (\rD_{\g} \pbar) \xi + w \eta' \xi$. This is in Atiyah-Patodi-Singer form $\partial_s + A(s)$ with $A(s)= J(\g(s,\cdot))\partial_t + F_{\g(s,\cdot)} + w\eta'(s){\rm Id}$, so that the Fredholm property follows from e.g.\ \cite[Thm.3.10]{robbinsalamon}. Here the additional term in $\lim_{s\to\pm\infty}A(s)= J(c_0)\partial_t + \text{\it Hess}_{c_0} \pm w {\rm Id}$ preserves invertibility of the limit operators since $w\in(-\d,\d)$ lies in the spectral gap of the self-adjoint operator $J(c_0)\partial_t + \text{\it Hess}_{c_0}$ on $H^0(S^1,\C^n)$.\footnote{Lemma~\ref{sc fred} abstractly shows that $\rD_{\g} \pbar : H^{m+1,\d_m} \to H^{m,\d_m}$ have the same Fredholm index for all $m\in\N_0$, but the transformation here shows that the choice of $\d_m$ in the spectral gap is crucial:
The Fredholm index of $\rD_{\g} \pbar : H^{1,w} \to H^{0,w}$ is the spectral flow of the self-adjoint operator family $A(s)= J(\g(s,\cdot))\partial_t + \rD_{\g(s,\cdot)}(JX) + w \eta'(s){\rm Id}$ by e.g.\ \cite[Thm.4.1]{robbinsalamon}. By homotopy invariance of the spectral flow, it is constant under variation of $w$ if the limit operators $J(c_0)\partial_t + \text{\it Hess}_{c_0} \pm w {\rm Id}$ remain invertible. 
}
\end{enumerate}
This shows that $\rD_{\B}\Phi(0,0) : (B_m)_{m\in\N_0} \to (E_m)_{m\in\N_0}$ is sc-Fredholm as in Definition~\ref{fred op}. By Lemma~\ref{sc fred} this transfers to $\rD_{\B}\Phi(0,0) :  (B_m)_{m\in\N} \to (E_m)_{m\in\N}$.

To verify the regularization property of $\rD_\B\Phi(r,0)$ for $r>0$ consider \eqref{defDPhi} 
\begin{align*}  
(\widehat\oplus_R \times \widehat\ominus_R) \; \rD_\B\Phi(r, 0) (\xi_1,\xi_2) 
&\;=\;
\left( \begin{aligned}
\rD_{\g_r} \pbar \oplus_R \! (\xi_1,\xi_2)  
\\
\rD_{\overline c_0} \pbar \ominus_R \! (\xi_1,\xi_2)  
\end{aligned}
\right) , 
\end{align*}
which expresses $\rD_\B\Phi(r,0)$ as a composition of the linearized Floer operators at $\g_r, \overline c_0$ -- which are regularizing by (ii) above -- and the isomorphisms $\oplus_R \times\ominus_R$ and $\widehat\oplus_R \times \widehat\ominus_R$. We claim that for $r>0$ and thus $R(r)<\infty$ these Cartesian products of pregluing and anti-pregluing operations induce isomorphisms 
$H^{k,w}\times H^{k,w} \;\overset{\sim}{\longrightarrow}\; H^{k,w}\times H^{k,w}$ for all $k\in\N_0$ and weights $w\in\R$. 
Indeed, both maps have the form
\begin{align*}
(\zeta_1,\zeta_2) \; \mapsto\; 
\left( \begin{aligned}
(\b-1) \cdot   \tau_R\,\zeta_1 +\b \cdot  \tau_{ - R} \,\zeta_2  \\
\b \cdot \tau_R \, \zeta_1 + (1-\b) \cdot  \tau_{ - R} \, \zeta_2
\end{aligned}\right)
\end{align*}
with inverse operators from \eqref{eq:totalglue} given by
\begin{align*}
(\zeta_1,\zeta_2) \; \mapsto\; 
\left( \begin{aligned}
 \tau_{\mi R}  \bigl( \tfrac{\b}{\b^2 + (1-\b)^2} \zeta_1 + \tfrac{\b-1}{\b^2 + (1-\b)^2} \zeta_2 \bigr)  \\
\tau_{R}  \bigl(\tfrac{1-\b}{\b^2 + (1-\b)^2} \zeta_1 + \tfrac{\b}{\b^2 + (1-\b)^2}  \zeta_2 \bigr) 
\end{aligned}\right) . 
\end{align*}
And these maps preserve $H^{k,w}$-regularity since $H^{k,w}$ is closed under the shift operations $\tau_{\pm R}$ for fixed $R$, as well as under sums and multiplication with smooth functions that have uniformly bounded derivatives.

Now consider $\xi_1,\xi_2\in H^{1,w}$ with $|w|<\d$ and $\rD_{\B}\Phi(r,0)(\xi_1,\xi_2)\in H^{m,\delta_m}\times H^{m,\delta_m}$. Then we have $\zeta_+:= \oplus_R (\xi_1,\xi_2)\in H^{1,w}$ with $\rD_{\g_r} \pbar\zeta_+\in H^{m,\delta_m}$ and $\zeta_-:= \ominus_R (\xi_1,\xi_2) \in H^{1,w}$ with $\rD_{\overline c_0} \pbar\zeta_-\in H^{m,\delta_m}$. 
Then the regularization property (ii) of the linearized Floer operators at $\g_r, \overline c_0$ implies 
$(\zeta_+,\zeta_-)\in H^{m+1,\delta_m}\times H^{m+1,\delta_m}$, and hence
$(\xi_1,\xi_2) = (\oplus_R \times \ominus_R)^{-1} (\zeta_+,\zeta_-)\in H^{m+1,\delta_m}\times H^{m+1,\delta_m}$. 


Finally, to see that $\rD_\B\Phi(r,0): H^{m+1,\d_m}\to H^{m,\d_m}$ is Fredholm\footnote{
Here we use $H^{m+1,\d_m}\to H^{m,\d_m}$ 
as short hand for an operator from $B_m=H^{m+1,\d_m}\times H^{m+1,\d_m}$ to $E_m = H^{m,\d_m}\times H^{m,\d_m}$. 
This should prevent confusion between the sc-Banach space $\B=(B_m)_{m\in\N_0}$ and the error terms $B_1,B_2$.   
} 
for sufficiently small $r>0$ (depending on $m\in\N_0$ and $|\d_m|<\d$) with the same index as $\rD_\B\Phi(0,0): H^{m+1,\d_m}\to H^{m,\d_m}$ (which is independent of $m\in\N_0$ by Lemma~\ref{sc fred}) we proceed in two steps:

Lemma~\ref{lem:est} compares the operators $\uxi\mapsto \rD_\B\Phi(r,0)\uxi - (B_1(r,\uxi),B_2(r,\uxi))$ and $\rD_\B\Phi(0,0)$ and bounds their difference in the operator norm (on the space of bounded operators from $H^{m+1,\d_m}$ to $H^{m,\d_m}$) by $\| \g_1 \|_{\cC^{m+1}([R-1,\infty)\times S^1) }$ and $ \|  \g_2 \|_{\cC^{m+1}((\mi \infty,\mi R+1]\times S^1) }$, both of which converge to zero as $r\to 0$. 
Since the set of Fredholm operators is open and the index is locally constant, this proves that there is $\eps_m>0$ so that $\rD_\B\Phi(r,0) - (B_1(r,\cdot),B_2(r,\cdot))$ is Fredholm with the index of $\rD_\B\Phi(0,0)$ for all $0<r<\eps_m$. 

Second, for any fixed $r<\eps_0$, we claim that $\rD_\B\Phi(r,0)$ is Fredholm with the same index as well because the difference $(B_1(r,\cdot),B_2(r,\cdot): H^{m+1,\d_m}\to H^{m,\d_m}$ is a compact operator. Indeed, it factors into the compact Sobolev embedding $H^{m+1,\d_m}\hookrightarrow H^{m,\d_{m-1}}$ for $\d_{m-1}<\d_m$ and  
a bounded operator $H^{m,\d_{m-1}}\to H^{m,\d_m}$ by the estimate 
\begin{align*}
& \| B_1(r,\xi_1,\xi_2) \|_{H^{m,\d_m}} +  \| B_2(r,\xi_1,\xi_2) \|_{H^{m,\d_m}} \\
&\qquad\leq C_m \bigl( \| \xi_1 \|_{H^{m,\d_m}([R-1,R+1]\times S^1)} + \| \xi_2 \|_{H^{m,\d_m}([\mi R-1,\mi R+1]\times S^1)} \bigr) \\
&\qquad\leq C_m  \| e^{(\d_m-\d_{m-1})\eta} \|_{\cC^m([\mi R-1,R+1])} \bigl( \| e^{\d_{m-1}\eta} \xi_1 \|_{H^m} + \|  e^{\d_{m-1}\eta}\xi_2 \|_{H^m} \bigr) \\
&\qquad\leq C e^{(\d_m-\d_{m-1})(R+1)} \bigl( \| \xi_1 \|_{H^{m,\d_{m-1}}} + \| \xi_2 \|_{H^{m,\d_{m-1}}} \bigr).
\end{align*}
Note that, while the constants in these estimates depend on $r>0$, they do prove compactness of the operators $(B_1(r,\cdot),B_2(r,\cdot))$ for all $r>0$ and thus yield the Fredholm property for $\rD_\B\Phi(r,0)$ with  index independent of $r\ge 0$. 
\end{proof}

\begin{remark}\rm \label{rmk:gen3.9}
Theorem~\ref{mainthm} extends to the general setting of Remarks~\ref{rmk:gen3}, \ref{rmk:gen4} by the following adjustments to the proof:
\begin{itemlist}
\item
The formula \eqref{DPhi} for the partial derivative $\rD_\B \Phi$ continues to hold if we replace $\R\times S^1\times \C^n$ with the pullback bundles $\g_\pm^*\rT M$, identify $c^*\rT M\simeq c_0^*\rT M$ for all $c\in\cC^\infty(S^1,M)$ close to $c_0$ as in Remark~\ref{rmk:gen1}, and use exponential maps instead of addition in the construction of $\Phi$ as in Remark~\ref{rmk:gen3}.
This only requires to adjust the notation to $\g_r^\pmi:=\tau_{\pmi R}\ti\g_R$ as in Remark~\ref{rmk:gen1} and $\g_r^{e\,\pmi}:=\tau_{\pmi R}\exp_{\ti\g_R}(e)$.
\item
The estimates in Lemma~\ref{le1}, \ref{lem:est}, \ref{le2}, \ref{lem:invert} continue to hold after replacing any $\cC^{m+1}$-norm of $\g_{\ldots} = \g_{\ldots}-\overline c_0$ with the $\cC^{m+1}$-norm of the section $(s,t)\mapsto (\exp_{\overline c_0(t)})^{-1}(\g_{\ldots}(s,t))$ of $\overline c_0^* \rT M$.
 
The proofs use pointwise estimates arising from the fact that in local trivializations the exponential differs from addition only by quadratic terms.
\item
The linearized Fredholm properties in Lemma~\ref{le3} directly transfer under the shift in regularity from $\B^0,\E^0$ to $(B_m)_{m\ge 2}, (E_m)_{m\ge 2}$. The generalization to pullback bundles does not affect statements or proofs, and the restriction of the base space to the codimension $2$ subspaces $B_m= \{\xi_\pm(0,0)\in \rT_{\g_\pm(0,0)}\Si_\pm\}\subset H^{m+1,\d_m}\times H^{m+1,\d_m}$ as in Remark~\ref{rmk:gen4} affects only the classical Fredholm property of $\rD_\B\Phi(r,0)$ for $r\ge 0$. 

For $r=0$ we argued that the linearized Floer operators $\rD_\g\pbar:H^{1,w}\to H^{0,w}$ are Fredholm and deduced -- via the other sc-Fredholm properties (i),(ii) in Definition~\ref{fred op} and Lemma~\ref{sc fred} -- that the operators in higher regularity $\rD_\g\pbar:H^{m+1,\d_m}\to H^{m,\d_m}$ are Fredholm. Then $\rD_\B\Phi(0,0)$ is given by the two restrictions $\rD_{\g_\pm}\pbar|_{\{\xi_\pm(0,0)\in \rT_{\g_\pm(0,0)}\Si_\pm\}}$. These are Fredholm with index reduced by the  codimension of the restriction.

For $r>0$ the comparison with $\rD_\B\Phi(0,0)$ via estimates (in the operator norm resp.\ showing compactness of extra terms) is not affected. 
\end{itemlist}

More abstractly, this discussion shows that restriction of a sc-Fredholm map $\Phi:[0,\eps)\times\B\to\E$ to higher scales $\B=(B_{\ell+m})_{m\in\N_0}$, $\E=(E_{\ell+m})_{m\in\N_0}$ does not affect the sc-Fredholm property or index. The restriction $\Phi:[0,\eps)\times\B'\to\E$ to an sc-complement $\B=\B'\oplus_{\rm sc} C$ of a finite dimensional subspace $C\subset B_\infty$ does not affect the sc-Fredholm property either, but reduces the index by $\dim C$. 
However, neither of these arguments can be used to deduce the sc-Fredholm property of $\Phi:[0,\eps)\times\B^{S^1}\to \E^{S^1}$ in the setup for Morse theory from Remarks~\ref{rmk:gen3}, \ref{rmk:gen4}, since the $S^1$-invariant function spaces $\B^{S^1}\subset \B$, $\E^{S^1}\subset\E$ have infinite codimension. 
However, we can directly transfer every step of the proof of the sc-Fredholm property in Theorem~\ref{mainthm} holds to the Morse theory setting by replacing properties of the linearized Floer operator with the analogous properties of the linearized gradient flow operator $\nabla_s + \rD_\g \nabla H: H^{m+1,w}(\R,\g^*\rT M)\to H^{m,w}(\R,\g^*\rT M)$. 
\end{remark}

%
%

\appendix

\section{Brief Guide to Polyfold Theory}
\label{sec:polyfold}

Polyfold theory was developed in \cite{hwzbook} in order to {\it regularize} moduli spaces of pseudoholomorphic curves. In the simplest cases, this is achieved by describing a Gromov-compactified moduli space $\CM=\s^{-1}(0)$ as the zero set of a section and associating to it a cobordism class $[(\s+p)^{-1}(0)]$ of perturbed zero sets obtained from the following core theorem.

\medskip\noindent
{\bf M-Polyfold Regularization Theorem:} 
Let $\cE\to\cB$ be a strong M-polyfold bundle modeled on scale Hilbert spaces, and let $\s:{\cB}\to{\cE}$ be a scale smooth Fredholm section such that $\s^{-1}(0)\subset{\cB}$ is compact.
Then there exists a class of sc$^+$ perturbation sections $\nu:{\cB}\to{\cE}$ supported near $\s^{-1}(0)$ such that $\s+\nu$ is transverse to the zero section. As a consequence, $(\s+\nu)^{-1}(0)$ carries the structure of a smooth compact manifold. Moreover, for any other such perturbation $\nu':{\cB}\to{\cE}$ there exists a smooth compact cobordism between $(\s+\nu')^{-1}(0)$ and $(\s+\nu)^{-1}(0)$.

\medskip
Similar polyfold regularization theorems exist for moduli spaces with isotropy as well as boundary and corners; in general yielding transverse multivalued perturbations, whose zero sets are compact weighted branched orbifolds with boundary and corners, unique up to appropriate cobordism. 
The following is a glossary for the new language used to formulate these theorems. More detailed introductions to the underlying ideas can be found in e.g.\  \cite[2.1]{theguide}.

\medskip
\noindent
{\bf Scale Banach / Hilbert space:} A {\it Banach / Hilbert space} such as  $L^2(S^1)$ with additional scale structure such as $(H^k(S^1))_{k\in\N_0}$. \\
{\it (In finite dimensions:  vector space with norm / inner product.)} 
 
\smallskip
\noindent
{\bf Scale differentiability / smoothness:} A notion of {\it differentiability / smoothness} for maps between scale Banach spaces such that (classically nonwhere differentiable) reparametrization actions such as $S^1 \times L^2(S^1) \to L^2(S^1), (t, u) \mapsto u(t+\cdot)$ are scale smooth, and the chain rule holds. \\
{\it (In finite dimensions:  classical differentiability / smoothness.)}
 
\smallskip
\noindent
{\bf M-polyfold:} Analogue to the notion of {\it Banach manifold}, designed to allow a notion of smooth structure on a space $\cB$ of (not necessarily pseudoholomorphic) maps modulo reparametrization whose (nodal and non-nodal) domains vary over a Deligne-Mumford type space. \\
{\it (In finite dimensions: manifold.)}

\smallskip
\noindent
{\bf Strong M-polyfold bundle:} Analogue to the notion of {\it Banach bundle} -- a bundle $\cE\to\cB$ over an M-polyfold $\cB$ whose fibers are scale Banach spaces. \\
{\it (In finite dimensions: vector bundle.)}

\smallskip
\noindent
{\bf Polyfold (with boundary/corners):} Analogue to the notion of {\it Banach orbifold (with boundary/corners)}, generalizing the notion of M-polyfold to allow for maps with nontrivial isotropy and Deligne-Mumford spaces with boundary/corners. \\
{\it (In finite dimensions: orbifold  (with boundary/corners).)}

\smallskip
\noindent
{\bf Strong polyfold bundle:} Analogue to the notion of {\it Banach orbi-bundle} -- a bundle $\cE\to\cB$ over a polyfold $\cB$ whose fibers are scale Banach spaces.\\ 
{\it (In finite dimensions: orbi-bundle.)}

\smallskip
\noindent
{\bf Scale-smooth Fredholm section of strong (M-)polyfold bundle:} Analogue to the notion of {\it Fredholm section in a Banach (orbi-)bundle}, which allows to describe a family of Cauchy-Riemann operators arising from complex structures in a Deligne-Mumford space (and almost complex structures possibly undergoing ``neck stretching'') as a single section $\sigma=\bar\partial_J:\cB\to\cE$ of an appropriate (M-)polyfold bundle.\\
{\it (In finite dimensions: smooth section.)}

\smallskip
\noindent
{\bf sc$^+$ section of strong (M-)polyfold bundle:} Analogue to the notion of {\it compact perturbation of a Fredholm section}. Examples are reparametrization-invariant $0$-th order perturbations of Cauchy-Riemann operators. \\
{\it (In finite dimensions: smooth section.)}

\smallskip
\noindent
{\bf Regularization with boundary and corners:} The polyfold regularization theorem generalizes directly to Fredholm sections $\sigma:\cB\to\cE$ over (M-)polyfolds with boundary and corners in various versions corresponding to the notion of transversality to the boundary strata and admissible perturbations (arising from compatibility requirements between perturbations and gluing/breaking).

General sc$^+$ perturbations can always be chosen such that $\s+\nu$ is ``neatly transverse''
and hence $(\s+\nu)^{-1}(0)$ is a compact manifold (resp.\ weighted branched orbifold) with boundary and corners, whose corner strata are given by its intersection with the corresponding boundary strata of $\cB$.

\end{document}